\documentclass[12pt]{amsart}
\usepackage{color}
\usepackage[dvipdfm]{graphicx}
\usepackage{ulem}
\usepackage{amsaddr}
\usepackage{url}

\def\disp{\displaystyle}
\def\R{\mathbb{R}}
\def\N{\mathbb{N}}
\def\C{\mathbb{C}}

\def\calF{\mathcal{F}}
\def\calH{\mathcal{H}}
\def\calL{\mathcal{L}}
\def\calZ{\mathcal{Z}}
\def\calU{\mathcal{U}}
\def\calV{\mathcal{V}}
\def\calX{\mathcal{X}}
\def\calR{\mathcal{R}}
\def\e{\varepsilon}

\def\pt{\partial_t}
\def\pR{\phi_{\rm{R}}}
\def\pI{\phi_{\rm{I}}}

\def\lambdaC{{\lambda_{\rm{R}}+i\lambda_{\rm{I}}}}
\def\lR{\lambda_{\rm R}}
\def\lI{\lambda_{\rm I}}
\def\lIn{\lambda_{{\rm I},n}}
\def\phiC{{\phi_{\rm{R}}+i\phi_{\rm{I}}}}
\def\PhiC{{\Phi_{\rm{R}}+i\Phi_{\rm{I}}}}

\def\XiC{{\Xi_{\rm{R}}+i\Xi_{\rm{I}}}}
\def\etaC{{\eta_{\rm{R}}}+i\eta_{\rm{I}}}
\def\Id{\rm{Id}}

\def\p{\partial}
\def\kn{k_n}
\def\k1{k_1}
\def\vphi{\varphi}
\def\bu{\bar{u}}
\def\bxi{\bar{\xi}}

\DeclareMathOperator{\sn}{sn}
\DeclareMathOperator{\sgn}{sgn}

\theoremstyle{plain}
	\newtheorem{theorem}{Theorem}[section]
	\newtheorem{lemma}[theorem]{Lemma}

	\newtheorem{proposition}[theorem]{Proposition}
	\newtheorem{remark}[theorem]{Remark}
	
\theoremstyle{plain}

\newcommand{\rf}[1]{(\ref{#1})}

\newcommand{\ddfrac}[2]{\displaystyle \frac{#1}{#2}}

\makeatletter
 \@addtoreset{equation}{section}
\makeatother
\def\theequation{\arabic{section}.\arabic{equation}}

\begin{document}
\title[Very slow layer oscillations]{
Exact solutions describing
very slow layer oscillations in
a shadow reaction-diffusion system}

\author{Shin-Ichiro Ei}
\thanks{SE was supported by JSPS KAKENHI Grant Number 22H01129.}

\address{Department of Mathematics, Josai University,\\
1-1 Keyakidai, Sakado, Saitama 350-0295, Japan\\
{\rm{\texttt{sei@josai.ac.jp}}}
}

\author{Yasuhito Miyamoto*}
\thanks{*Corresponding author}
\thanks{ORCiD of YM is 0000-0002-7766-1849}
\thanks{YM was supported by JSPS KAKENHI Grant Number 24K00530.}
\address{Graduate School of Mathematical Sciences, The University of Tokyo,\\
3-8-1 Komaba, Meguro-ku, Tokyo 153-8914, Japan\\
{\rm{\texttt{miyamoto@ms.u-tokyo.ac.jp}}}
}

\author{Tatsuki Mori}
\thanks{TM was supported by JSPS KAKENHI Grant Number 22K13962.}
\address{
Faculty of Engineering, Musashino University,\\
3-3-3 Ariake, Koto-ku, Tokyo 135-8181, Japan\\
{\rm{\texttt{t-mori@musashino-u.ac.jp}}}
}

\begin{abstract}
We show in a rigorous way that a stable internal single-layer stationary solution is destabilized by the Hopf bifurcation as the time constant exceeds a certain critical value.
Moreover, the exact critical value and the exact period of oscillatory solutions can be obtained.
The exact period indicates that the oscillation is very slow, {\it i.e.,} the period is of order $O(e^{C/\e})$.
We also rigorously prove that Hopf bifurcations from multi-layer stationary solutions occur.
In this case anti-phase horizontal oscillations of layers are shown by formal calculations.
Numerical experiments show that the exact period agrees with the numerical period of a nearly periodic solution near the Hopf bifurcation point.
Anti-phase (out of phase) horizontal oscillations of layers are numerically observed.
\end{abstract}
\date{\today}
\maketitle

\section{Introduction and main results}

We are concerned with the shadow reaction-diffusion system
\begin{equation}\label{S1E1}
\begin{cases}
\pt u=\e^2u_{xx}+f(u)-\alpha\xi & \textrm{for}\ 0<x<1,\ 0<t<T,\\
\disp \tau\pt\xi=\int_0^1(\beta u-\gamma\xi)dx & \textrm{for}\ 0<t<T,\\
u_x(0,t)=u_x(1,t)=0 & \textrm{for}\ 0<t<T,
\end{cases}
\end{equation}
where $u=u(x,t)$ and $\xi=\xi(t)$ are unknown functions, $\e$, $\tau$, $\alpha$, $\beta$ and $\gamma$ are positive constants and $f$ is defined by
\begin{equation}\label{f}
f(u):=u-u^3.
\end{equation}
\begin{remark}
Coefficients of \rf{S1E1}  can be reduced to $\alpha = \beta = 1$ by the normalization
$\alpha\xi \to \xi$, $\ddfrac{\tau}{\alpha\beta} \to \tau$, 
$\ddfrac{\gamma}{\alpha\beta} \to \gamma$. However, we deal with the original equation \rf{S1E1}
as it is without normalization for the sake of the clarification of the relation between coefficients.
\end{remark}
This system is obtained by using the shadow limit $D\to\infty$ of the reaction-diffusion equation
\begin{equation}\label{S1E2}
\begin{cases}
\pt U=\e^2U_{xx}+f(U)-\alpha V & \textrm{for}\ 0<x<1,\ 0<t<T,\\
\tau\pt V=DV_{xx}+\beta U-\gamma V & \textrm{for}\ 0<x<1,\ 0<t<T,\\
U_x(0,t)=U_x(1,t)=V_x(0,t)=V_x(1,t)=0 & \textrm{for}\ 0<t<T,
\end{cases}
\end{equation}
where $U=U(x,t)$ and $V=V(x,t)$ are unknown functions.
Taking the limit $D\to\infty$, one can expect that $V$ spreads over the domain $0<x<1$ quickly, and hence $V$ becomes a spatially homogeneous function depending only on $t$ which is denoted by $\xi(t)$.
Replacing $\xi$ instead of $V$ and integrating the second equation of (\ref{S1E2}), we obtain (\ref{S1E1}).
The problem (\ref{S1E2}) is a mathematical model arising in various branches of mathematical sciences, {\it e.g.}, pattern formations in morphogenesis and chemical reaction processes.
In morphogenesis $U$ and $V$ stand for concentrations of biochemicals called {\it the activator} and {\it inhibitor}, respectively.
Therefore, (\ref{S1E2}) is often called the activator-inhibitor system.
It is known that (\ref{S1E2}) exhibits Turing patterns, {\it i.e.}, spontaneous spatial temporal patterns.

Theoretically for the system \rf{S1E2}, the front type solution with
internal layers have been constructed mainly by the singular perturbation technique
using the smallness of the parameter $\e$.
And also the stability of stationary solutions with internal layers 
was extensively investigated by the method of the singular limit eigenvalue problem (SLEP) 
(\cite{NF87}) and the geometrical singular perturbation methed (\cite{J95}).
 Through such researches, the complicated motion of multi-layer solutions
 have been known. 
 One of the typical example is an oscillating behavior of two layer solutions. 
 In practice, \cite{KK} showed the oscillating interaction of two layers
 and \cite{NM89} showed the occurrence of the Hopf bifurcation. 

Since the shadow system (\ref{S1E1}) is a simplified system of (\ref{S1E2}) 
as the limit $D \to \infty$,
there are several properties in \rf{S1E1} 
inherited from \rf{S1E2} while some properties are not inherited.
One of the typical examples which are not inherited is that
 \rf{S1E1} has only monotone stationary solutions as stable ones
while \rf{S1E2} has stable stationary solutions with multi-layers.
And also two layer solutions in \rf{S1E1} can move with exponentially slow
speed called very slow motion (\cite{KEM92}) while no such motion
does not appear in \rf{S1E2}.
Similarly, the side to side oscillating behavior of two layers has been 
believed as the non-inherited property of \rf{S1E1} from \rf{S1E2}
while the up and down movement  of a spike solution of the shadow
Gierer-Meinhardt model was reported in \cite{WW03a}. In practice,
up and down oscillations frequently appear in ordinaray differential systems
by changing the time coefficient like $\tau$ in \rf{S1E1} and \rf{S1E2}.
In that sense, up and down oscillations are regarded as the motion by ODE structures.
On the other hand, the occurrence of the side to side oscillating behavior
is essentially due to PDE structures such as boundary conditions and diffusions,
which is  shown for \rf{S1E1} in this paper by using exact solutions. 
Actually, we solve \rf{S1E1} exactly and   
give an exact single and multi-layer stationary solutions.
By analyzing the corresponding eigenvalue problem, we show 
that the stable single layer stationary solution can be
destabilized in the sense of the Hopf bifurcation
by making $\tau$ so larger as $O(e^{C/\e})$ while
the multi-layer stationary solution always has exponentially small positive eigenvalues,
which indicates the instability and the very slow dynamics corresponding to
the result of \cite{KEM92}.  
It is also shown that changing $\tau$ so larger as $O(e^{C/\e})$ 
in \rf{S1E1} causes
 the Hopf bifurcation of the multi-layer stationary solution.

Thus, it is understood that the origin of the oscillating behavior in multi-layer 
solutions can be traced back to the oscillation occurring within a single-layer 
solution while it has been believed that oscillations were caused 
by the interaction between two or more layers in general. 
Therefore, the results presented in this paper,
in particular, the Hopf bifurcation of the single layer stationary solution
 offer unintuitive and unpredictable 
insights into the dynamics of layers in shadow systems.
\bigskip

In the following, we give more precise explanations.
Let $n\in\N:=\{1,2,\ldots\}$.
The problem (\ref{S1E1}) has internal $n$-layer stationary solutions and they can be written as
\begin{equation}\label{un}
(\e,u^{\pm}_n(x),\xi)=\left(
\frac{1}{2n\sqrt{1+k^2}K(k)},
\pm\sqrt{\frac{2k^2}{1+k^2}}\sn \left((2nx+1)K(k),k\right),0\right)
\end{equation}
for $0<k<1$.
Here, $\sn(x,k)$ denotes Jacobi's elliptic function and $K(k)$ denotes the complete elliptic integral of the first kind.
Definitions and useful properties are recalled in Appendix of the present paper.
The parameter $k$ is uniquely determined by $\e$ through the relation
\begin{equation}\label{e}
\sqrt{1+k^2}K(k)=\frac{1}{2n\e}\ \ \textrm{for}\ \ 0<k<1
\end{equation}
provided that $\e\in (0,1/n\pi)$.
The solution $k$ of (\ref{e}) is denoted by $\kn$.
We see that $k_n\to 1$ as $\e\to 0$.
The layers of $u^{\pm}_n$ become sharper as $\e\to 0$.
The stationary solution (\ref{un}) is independent of $\alpha$, $\beta$ and $\gamma$.
We frequently use the relation:
\[
\rho:=u^{\pm}_n(0)=\pm\sqrt{\frac{2\kn^2}{1+\kn^2}}.
\]

Let
\[
X:=H^2_N(0,1)\times\R\ \ \textrm{and}\ \ Y:=L^2(0,1)\times\R,
\]
where $L^2(0,1)$ denotes the usual complex Lebesgue space on $(0,1)$ equipped with the inner product $\left<\,\cdot\,,\,\cdot\,\right>$,
\[
H^2_N(0,1):=\left\{u(x)\in H^2(0,1);\ u'(0)=u'(1)=0\right\},
\]
the prime stands for the derivative with respect to $x$ and $H^2(0,1)$ denotes the usual complex Sobolev space of order $2$.
Here, the $L^2$-inner product $\left<\,\cdot\,,\,\cdot\,\right>$ is defined by
\[
\left<u(x),v(x)\right>=\int_0^1u(x)\overline{v(x)}dx\ \ \textrm{for}\ \ u,v\in L^2(0,1)
\]
and the overline indicates complex conjugate.
Hence, it satisfies the linearity in the first argument and the conjugate symmetry, {\it i.e.},
\[
\left<au+bv,w\right>=a\left<u,w\right>+b\left<v,w\right>\ \textrm{for all}\ a,b\in \C\ \textrm{and all}\ u,v,w\in L^2(I),
\]
\[
\left<u,v\right>=\overline{\left<v,u\right>}\ \textrm{for all}\ u,v\in L^2(I).
\]

The linearized eigenvalue problem of (\ref{S1E1}) around (\ref{un}) becomes
\begin{equation}\label{EVP}
\begin{cases}
L[\phiC]-\alpha(\etaC)=(\lambdaC)(\phiC) & \textrm{for}\ 0<x<1,\\
\frac{\beta}{\tau}\left<\phiC,1\right>-\frac{\gamma}{\tau}(\etaC)=(\lambdaC)(\etaC),\\
\pR'(0)=\pI'(0)=\pR'(1)=\pI'(1)=0,
\end{cases}
\end{equation}
where
\begin{equation}\label{L}
L:=\e^2\partial_{xx}+f'(u^{\pm}_n(x))
\end{equation}
defined on $H^2_N(0,1)$.
The linear operator on $X$ defined by the left hand side
 of (\ref{EVP}) is denoted by
\[
\calL:=\left(
\begin{array}{cc}
L & -\alpha\\
\frac{\beta}{\tau}\left<\,\cdot\,,1\right> & -\frac{\gamma}{\tau}
\end{array}
\right).
\]
$\sigma_{\calL}$ and $\sigma_L$ denote the spectral sets of $\calL$ and $L$, respectively.
They consist only of eigenvalues.
In this paper a complex eigenvalue means a non-real eigenvalue.
Let
\[
\calH_+:=\{\lambda\in\C;\ {\rm Re}\lambda>0\},\ \ 
\calH_-:=\{\lambda\in\C;\ {\rm Re}\lambda<0\},
\]
and let $i\R$ denote the imaginary axis and the pure imaginary numbers.

The first main result is about single-layer solutions $u^{\pm}_1$.
\begin{theorem}\label{T1}
Suppose that $0<\e<1/\pi$.
Let us consider single-layer stationary solutions $(u^{\pm}_1,0)$ and let $\k1$ be defined by (\ref{e}) with $n=1$.
Let $\delta:=\alpha\beta/\gamma$ and
\begin{multline}\label{T1E-1}
\tau_n={2(\delta+2)\gamma}\Bigg/\left[\left\{
\left(2(\delta+2)-\frac{3\delta}{1+\kn^2}\left(2\frac{E(\kn)}{K(\kn)}-1+\kn^2\right)\right)^2
\right.\right.\\
\left.\left.+24(\delta+2)\left(\frac{1-\kn^2}{1+\kn^2}\right)^2
\right\}^{1/2}
-2(\delta+2)+\frac{3\delta}{1+\kn^2}\left(2\frac{E(\kn)}{K(\kn)}-1+\kn^2\right)\right],
\end{multline}
where $K(k)$ and $E(k)$ are Jacobi's complete elliptic integrals of the first and second kinds, respectively.
They are defined by (\ref{K}) and (\ref{E}) in Appendix.
Let $\chi_n$ be defined by
\begin{equation}\label{delta}
\chi_n:=\frac{(1-\kn^2)^2K(\kn)}{(1+\kn^2)\left\{2E(\kn)-(1-\kn^2)K(\kn)\right\}}.
\end{equation}
Here, $\chi_n>0$ because of Lemma~\ref{AL1}.
Suppose that
\begin{equation}\label{T1E0}
\chi_1<\delta.
\end{equation}
Note that (\ref{T1E0}) holds if $\e>0$ is small ($k_1$ is close to $1$ and the layer becomes sharp).
Then there exists a unique $\tau_1>0$, which is given by (\ref{T1E-1}) with $n=1$, such that the following (i)--(iii) hold:\\
(i) If $0<\tau<\tau_1$, then $\sigma_{\calL}\subset \calH_-$. Hence, $(u^{\pm}_1,0)$ is stable.\\
(ii) If $\tau>\tau_1$, then $\sigma_{\calL}\cap \calH_+\neq\emptyset$. Hence, $(u^{\pm}_1,0)$ is unstable.\\
(iii) If $|\tau-\tau_1|$ is small, then $\sigma_{\calL}$ has exactly one pair of complex conjugate eigenvalues $\lR(\tau)\pm i\lI(\tau)$ such that $\lR(\tau)\pm i\lI(\tau)$ are simple, other eigenvalues are real and negative, both $\lR(\tau)$ and $\lI(\tau)$ are $C^1$-functions of $\tau$,
\[
\lR(\tau_1)=0,\ \ \lI(\tau_1)=\lambda_{{\rm I},1}>0\ \ \textrm{and}\ \ \frac{d\lR}{d\tau}(\tau_1)>0\ \textrm{(transversality condition)},
\]
where $\lambda_{\rm{I},1}$ is defined by the following (\ref{T1E1}) with $n=1$:
\begin{multline}\label{T1E1}
\lIn=\frac{1}{\sqrt{2}}\left[\left\{\left(2(\delta+2)-\frac{3\delta}{1+\kn^2}\left(2\frac{E(\kn)}{K(\kn)}-1+\kn^2\right)\right)^2
+24(\delta+2)\left(\frac{1-\kn^2}{1+\kn^2}\right)^2\right\}^{1/2}\right.\\
\left.-\left\{
2(\delta+2)-\frac{3\delta}{1+\kn^2}\left(2\frac{E(\kn)}{K(\kn)}-1+\kn^2\right)+6\left(\frac{1-\kn^2}{1+\kn^2}\right)^2\right\}\right]^{1/2}.
\end{multline}
Therefore, a Hopf bifurcation occurs at $\tau=\tau_1$.
\begin{itemize}
\item[(HB)]
Specifically, there exist $r_0>0$ and a continuously differential curve $\{(u(r),\xi(r),\tau(r))\}$, $|r|<r_0$, of $2\pi/\kappa(r)$-periodic solutions of (\ref{S1E1}) through $(u(0),\xi(0),\tau(0))=(u^{\pm}_1,0,\tau_1)$ with $2\pi/\kappa(0)=2\pi/\lambda_{\rm{I},1}$.
Every other periodic solution of (\ref{S1E1}) in a neighborhood of $(u^{\pm}_1,0,\tau_1)$ is obtained by a phase shift $S_{\theta}(u(r),\xi(r))$, $0\le\theta<2\pi/\kappa(\tau)$, of $(u(r),\xi(r),\tau(r))$.
In particular, $(u(-r),\xi(-r))=S_{\pi/\kappa(r)}(u(r),\xi(r))$, $\kappa(-r)=\kappa(r)$, $\tau(-r)=\tau(r)$ for all $|r|<r_0$ and
\[
\frac{2\pi}{\kappa(r)}=\frac{2\pi}{\lambda_{{\rm I},1}}+o(1)\ \ \textrm{as}\ \ r\to 0.
\]
\end{itemize}
\noindent
(iv) The following asymptotic formulas hold:
As $\e\to 0$,
\begin{align*}
\frac{2\pi}{\lambda_{{\rm I},1}}&=\frac{\pi}{12}\sqrt{\frac{\delta+2}{\delta}}
\frac{1}{\left({\sqrt{2}\e}\right)^{1/2}}\exp\left(\frac{1}{\sqrt{2}\e}\right)(1+o(1)),\\
\tau_1&=\frac{\gamma(\delta+2)}{192}\exp\left(\frac{\sqrt{2}}{\e}\right)(1+o(1)),\\
\chi_1&=\frac{16\sqrt{2}}{\e}\exp\left(-\frac{\sqrt{2}}{\e}\right)(1+o(1)).
\end{align*}
\end{theorem}
Note that the statement (HB) is taken from \cite[Theorem I.8.2]{K04}.
See Theorem~\ref{S5T1} of the present paper for the abstract Hopf bifurcation theorem.
Theorem~\ref{T1} says that a single-layer stationary solution is destabilized by the Hopf bifurcation as $\tau$ exceeds $\tau_1$ and that the oscillatory solution is very slow, {\it i.e.,} the period is $2\pi/\lambda_{{\rm I},1}=O(e^{C/\e})$ as $\e\to 0$.

We need some notation to state the next theorem.
We consider the scalar problem
\begin{equation}\label{SFE}
\begin{cases}
\e^2u''+f(u)=0 & \textrm{for}\ 0<x<1,\\
u'(0)=u'(1)=0,
\end{cases}
\end{equation}
where $f$ is defined by (\ref{f}).
Then, $u^{\pm}_n$ defined by (\ref{un}) are also solutions of (\ref{SFE}).
The associated eigenvalue problem of (\ref{SFE}) is
\begin{equation}\label{EVPSFE}
\begin{cases}
L\psi=\mu\psi & \textrm{for}\ 0<x<1,\\
\psi'(0)=\psi'(1)=0,
\end{cases}
\end{equation}
where $L$ ie defined by (\ref{L}).
Since $f'(u^+_n)\equiv f'(u^-_n)$, the eigenvalues of (\ref{EVPSFE}) for both $u^+_n$ and $u^-_n$ are the same.
Let $\sigma_L:=\{ \mu_j^{(n)}\}_{j=0}^{\infty}$,
\[
\mu^{(n)}_0>\mu^{(n)}_1\ge\mu^{(n)}_2\ge\cdots,
\]
denote the eigenvalues of (\ref{EVPSFE}).
Then it is known that the first eigenvalue satisfies
\begin{equation}\label{SE}
\mu^{(n)}_0=96\exp\left(-\frac{\sqrt{2}}{n\e}\right)(1+o(1))\ \ \textrm{as}\ \ \e\to 0.
\end{equation}
Detailed information on $\mu^{(n)}_j$, $j\ge 0$, can be found in Proposition~\ref{S2P1} of the present paper.

When the linearized operator has an eigenvalue with positive real part, the corresponding solution is unstable.
However, if the real part is small, {\it e.g.,} it is of order $O(\e^{-C/\e})$ as $\e\to 0$, then the solution is often called {\it metastable}.
The following theorem asserts that interior multi-layer stationary solutions are metastable for $0<\tau<\tau_n$ and the Hopf bifurcation occurs at $\tau=\tau_n$.
\begin{theorem}\label{T2}
Let us consider internal $n$-layer stationary solutions $(u^{\pm}_n,0)$, $n\ge 2$, and let $\kn$ be defined by (\ref{e}).
Let $\delta:=\alpha\beta/\gamma$ and $\chi_n$ be defined by (\ref{delta}).
Assume that
\begin{equation}\label{T2E0}
\chi_n<\delta.
\end{equation}
Note that (\ref{T2E0}) holds if $\e>0$ is small ($k_n$ is close to $1$ and the layers become sharp.).
Then there exists a unique $\tau_n>0$, which is defined by (\ref{T1E-1}), such that the following (i)--(iii) hold:\\
(i) For $\tau>0$, $\mu_1^{(n)}\in\sigma_{\calL}$, and hence $(u^{\pm}_n,0)$ are unstable.
Furthermore, if $0<\tau<\tau_n$, then $\sigma_{\calL}\cap\left(i\R\cup \calH_+\right)\subset (0,\mu^{(n)}_1]$.
Since $0<\mu^{(n)}_1<\mu^{(n)}_0$, it follows from (\ref{SE}) that $(u^{\pm}_n,0)$ are metastable for $0<\tau<\tau_n$.\\
(ii) If $\tau>0$ is large, then $\sigma_{\calL}\subset (-\infty,\mu^{(n)}_0]$.
Hence, it follows from (\ref{SE}) that $(u^{\pm}_n,0)$ are metastable for large $\tau>0$.\\
(iii) If $|\tau-\tau_n|$ is small, then $\sigma_{\calL}$ has exactly one pair of complex conjugate eigenvalues $\lR(\tau)\pm i\lI(\tau)$ such that $\lR(\tau)\pm i\lI(\tau)$ are simple, other eigenvalues are real and on $(-\infty,\mu^{(n)}_1]$,
$\lR(\tau)$ and $\lI(\tau)$ are $C^1$-functions of $\tau$,
\[
\lR(\tau_n)=0,\ \ \lI(\tau_n)=\lambda_{{\rm I},n}>0\ \ \textrm{and}\ \ \frac{d\lR}{d\tau}(\tau_n)>0\ \textrm{(transversality condition)},
\]
where $\lIn$ is given by (\ref{T1E1}).
Therefore, a Hopf bifurcation (HB) with $u^{\pm}_1$, $\tau_1$ and $\lambda_{I,1}$, which is stated in Theorem~\ref{T1}~(iii), replaced by $u^{\pm}_n$, $\tau_n$ and $\lambda_{I,n}$ occurs at $\tau=\tau_n$.
\\
(iv) The following asymptotic formulas hold:
As $\e\to 0$,
\begin{align}
\frac{2\pi}{\lIn}&=\frac{\pi}{12}\sqrt{\frac{\delta+2}{\delta}}
\frac{1}{\left({\sqrt{2}n\e}\right)^{1/2}}\exp\left(\frac{1}{\sqrt{2}n\e}\right)(1+o(1)),\label{T2F1}\\
\tau_n&=\frac{\gamma(\delta+2)}{192}\exp\left(\frac{\sqrt{2}}{n\e}\right)(1+o(1)),\label{T2F2}\\
\chi_n&=\frac{16\sqrt{2}}{n\e}\exp\left(-\frac{\sqrt{2}}{n\e}\right)(1+o(1)).\label{T2F3}
\end{align}
\end{theorem}
\bigskip

Jacobi's elliptic functions cannot be used for a wide class of nonlinear ODEs.
However, if they can be used, then those enable us to make detailed analysis.
In the study of (\ref{S1E1}) internal $n$-layer stationary solutions can be written in terms of elliptic functions, since the nonlinear term consists only of a cubic term.
Moreover, detailed informations on $L$ which are shown in Section~2 are already known.
Using these information, we show in Lemmas~\ref{L1} and \ref{L1+} that $\calL$ has at most two complex eigenvalues, which are conjugate, and that they can be obtained as roots of a cubic polynomial (\ref{g}).
In Lemma~\ref{L2} we determine the exact critical value of $\tau$ such that two complex eigenvalues are on $i\R\setminus\{0\}$ and obtain exact expressions of those eigenvalues.
The simplicity of pure imaginary eigenvalues (Lemma~\ref{L2+}) and the transversality condition (Lemma~\ref{L2}~(iii)) can be verified by explicit calculations of elliptic functions with an exact expression of a resolvent term (\ref{S2P2E0}).
In summary all important calculations can be carried out.
It is in general difficult to obtain the exact period and exact critical value in Hopf bifurcations.
Our case seems rare.
However, a similar analysis is applicable to the shadow Gierer-Meinhardt system
\[
\begin{cases}
\partial_tu=\e^2u_{xx}-u+\frac{u^p}{\xi^q} & \textrm{for}\ 0<x<1,\ 0<t<T,\\
\tau\partial_t\xi=\int_0^1\left(-\xi+\frac{u^r}{\xi^s}\right)dx & \textrm{for}\ 0<x<1,\\
u_x(0,t)=u_x(1,t)=0 & \textrm{for}\ 0<t<T
\end{cases}
\]
only when $(p,r)=(3,1)$ or $(3,3)$.
See \cite{MNN24} for details.
When $(p,r)=(3,2)$ or $(2,2)$, in \cite{WW03} a different method using Jacobi's elliptic functions was used to prove Hopf bifurcations.
However, exact periods and critical values were not obtained.
See also \cite{D01,GMW21,WW03a} for Hopf bifurcations for small $\e>0$.\\

This paper consists of eight sections.
In Section~2 we recall known results about eigenvalues of (\ref{EVPSFE}).
We also recall useful lemmas in our study of eigenvalues of a linear operator with nonlocal term.
In Section~3 we study eigenvalues of $\sigma_{\calL}$ in details.
In Section~4 we obtain asymptotic formulas (\ref{T2F1}), (\ref{T2F2}) and (\ref{T2F3}).
In Section~5 we prove Theorems~\ref{T1} and \ref{T2}, using lemmas proved in Sections~3 and 4.
In Section~6 we study the oscillatory behavior of multi-layer solutions, using formal calculations.
Specifically, anti-phase horizontal oscillations occur near the Hopf bifurcation point.
In Section~7 we show numerical simulations for (\ref{S1E1}).
In particular, we will see that the exact period $2\pi/\lambda_{{\rm I},1}$ agrees with a numerical period of a nearly periodic solution of \eqref{S1E1}.
In our numerical experiment the relative error percentage is
\[
1-\frac{\textrm{Numerical period}}{\textrm{Exact period}}\simeq 2.1\%.
\]
Moreover, anti-phase horizontal oscillations of layers are numerically observed.
The last section is the Appendix.
We recall the definition and properties of Jacobi's elliptic function $\sn(x,k)$ and the complete elliptic integral of the first kind $K(k)$ and of the second kind $E(k)$.
We also recall useful lemmas about $\sn(x,k)$ and $K(k)$.

\section{Preliminaries}
Let $L$ be the operator defined by (\ref{L}) in $L^2(0,1)$ with the domain $H_N^2(0,1)$.
\begin{proposition}\label{S2P1}
Let $\{\mu_j^{(n)}\}_{j=0}^{\infty}$ denote the eigenvalues of (\ref{EVPSFE}).
Then, the following (i)--(iv) hold:\\
(i) The eigenvalues $\mu^{(n)}_j$, $j\ge 0$, are simple and
\begin{equation}\label{S2P1E0}
\mu^{(n)}_0>\mu^{(n)}_1>\cdots >\mu^{(n)}_{n-1}>0>\mu^{(n)}_n>\mu^{(n)}_{n+1}>\cdots.
\end{equation}
Moreover, $\mu^{(n)}_0$ and $\mu^{(n)}_n$ are given as the roots of $\mu^2+2\mu-3(1-\rho^2)^2=0$, {\it i.e.,}
\begin{equation}\label{S2P1E1}
\mu^{(n)}_0=-1+\sqrt{1+3(1-\rho^2)^2}\ \ \textrm{and}\ \ \mu^{(n)}_n=-1-\sqrt{1+3(1-\rho^2)^2}.
\end{equation}
(ii) Let $\psi^{(n)}_j(x)$, $j\ge 0$, denote an eigenfunction associated with $\mu^{(n)}_j$.
Then $\psi^{(n)}_j$, $j=0,2,4,\ldots$, are symmetric and $\psi^{(n)}_j$, $j=1,3,5,\ldots$, are anti-symmetric.
Here, we call $\psi(x)$ symmetric if $\psi(x)=\psi(1-x)$ for $0\le x\le 1$, and call $\psi(x)$ anti-symmetric if $\psi(x)=-\psi(1-x)$ for $0\le x\le 1$.\\
(iii) The following asymptotic formulas hold: As $\e\to 0$,
\begin{align*}
&\mu^{(n)}_j=96\left(\cos^2\frac{j\pi}{2n}\right)\exp\left(-\frac{\sqrt{2}}{n\e}\right)(1+o(1)),&&\textrm{for}\ 0\le j<n,\ \\
&\mu^{(n)}_j=-\frac{3}{2}+12\left(\cos\frac{(j-n)\pi}{n}\right)\exp\left(-\frac{1}{\sqrt{2}n\e}\right)(1+o(1)),&&\textrm{for}\ n\le j<2n,\ \\
&\mu^{(n)}_{2n}=-2-96\exp\left(-\frac{\sqrt{2}}{n\e}\right)(1+o(1)),&& {} \\
&\mu^{(n)}_j=-2-(j-2n)^2\pi^2\e^2(1+o(1)),&&\textrm{for}\ j>2n.\ 
\end{align*}
\end{proposition}
\begin{proof}
It is well known that $\mu^{(n)}_j$, $j\ge 0$, becomes simple because of the uniqueness of the solution of the ODE in (\ref{EVPSFE}).
The inequalities (\ref{S2P1E0}) follow from the monotonicity of a time-map for (\ref{SFE}) and Sturm's comparison theorem.
We omit the details.
The exact eigenvalues (\ref{S2P1E1}) are given in \cite[Theorem 1]{W06}.
The assertion (ii) follows from an easy symmetric argument, since $f'(u^{\pm}_n)$ is symmetric.
The assertion (iii) was proved in \cite[Theorems 1.1 and 1.2]{WY15}.
\end{proof}

In the proof of main theorems we study the linear operator with nonlocal term.
\begin{proposition}\label{S2P5}
Let $A[\,\cdot\,]:=\left<\,\cdot\,,\Phi\right>\Psi$ for some $\Phi,\Psi\in L^2(0,1)$.
We define $h(\lambda)$ by
\begin{equation}\label{S2P5E1}
h(\lambda):=1+\left<(L-\lambda)^{-1}[\Psi],\Phi\right>\ \ \textrm{for}\ \ \lambda\in\C\setminus\sigma_L.
\end{equation}
(i) If $\lambda\in\C\setminus\sigma_L$ and $h(\lambda)\neq 0$, then $L+A-\lambda$ is invertible, and
\begin{equation}\label{S2P5E2}
(L+A-\lambda)^{-1}=\left({\rm Id}-\frac{(L-\lambda)^{-1}A}{h(\lambda)}\right)(L-\lambda)^{-1}.
\end{equation}
(ii) If $\lambda\in\C\setminus\sigma_L$ and $h(\lambda)=0$, then $\lambda$ is an eigenvalue of $L+A-\lambda$, the geometric multiplicity of $\lambda$ is one and the corresponding eigenfunction is $(L-\lambda)^{-1}[\Psi]$.
\end{proposition}
\begin{proof}
(i) The formula (\ref{S2P5E2}) with (\ref{S2P5E1}) is called the Sherman-Morrison formula.
We can check that if $\lambda\in\C\setminus\sigma_L$ and $h(\lambda)\neq 0$, then
\begin{align*}
\left({\rm Id}-\frac{(L-\lambda)^{-1}A}{h(\lambda)}\right)(L-\lambda)^{-1}(L+A-\lambda)=\Id,\\
(L+A-\lambda)\left({\rm Id}-\frac{(L-\lambda)^{-1}A}{h(\lambda)}\right)(L-\lambda)^{-1}=\Id.
\end{align*}
(ii) We consider the equation $(L+A-\lambda)\phi=0$. Then
\[
(L-\lambda)\phi+\left<\phi,\Phi\right>\Psi=0.
\]
Since $\lambda\not\in\sigma_L$, we have $\phi=-\left<\phi,\Phi\right>(L-\lambda)^{-1}[\Psi]$.
Therefore, the eigenfunction should be $\phi=c(L-\lambda)^{-1}[\Psi]$, $c\in\C$.
Thus the geometric multiplicity is one.
Moreover, when $\phi=c(L-\lambda)^{-1}[\Phi]$, we have
\[
(L+A-\lambda)\phi=c\left(1+\left<(L-\lambda)^{-1}[\Psi],\Phi\right>\right)\Psi=0,
\]
where we use $h(\lambda)=0$.
Thus, $\phi$ is an eigenfunction.
\end{proof}
Algebraic multiplicities of complex eigenvalues are studied in Lemma~\ref{L2+}.

Propositions~\ref{S2P2} and \ref{S2P3} are important to calculate $h(\lambda)$ defined by (\ref{S2P5E1}).
\begin{proposition}\label{S2P2}
Let $\lambda\in\C\setminus\sigma_L$.
Then, $\phi:=(L-\lambda)^{-1}[1]$ can be written explicitly as follows:
\begin{equation}\label{S2P2E0}
\phi=(L-\lambda)^{-1}[1]=\frac{-(3+\lambda)+3(u^{\pm}_n)^2}{\lambda^2+2\lambda-3(1-\rho^2)^2}.
\end{equation}
\end{proposition}
\begin{proof}
Let $u:=u^{\pm}_n$ for simplicity.
Multiplying (\ref{SFE}) by $u'$ and integrating it over $[0,x]$ we have
\[
\frac{\e^2u'^2}{2}+\frac{u^2}{2}-\frac{u^4}{4}=\frac{\rho^2}{2}-\frac{\rho^4}{4},
\]
and hence
\begin{equation}\label{S2P2E1}
L[u^2]=2\e^2uu''+2\e^2u'^2+u^2-3u^4=2\rho^2-\rho^4-3u^2.
\end{equation}
Using $L[1]=1-3u^2$ and (\ref{S2P2E1}), by direct calculation we can check that $(L-\lambda)\phi=1$.
Note that the denominator of (\ref{S2P2E0}) is not zero because
\[
\lambda\not\in\{\mu^{(n)}_0,\mu^{(n)}_n\}\iff \lambda^2+2\lambda-3(1-\rho^2)^2\neq 0.
\]
\end{proof}

\begin{proposition}\label{S2P3}
Let $u^{\pm}_n$ be given by (\ref{un}) and let $k_n$ be given by (\ref{e}).
Then,
\[
\int_0^1u^{\pm}_n(x)^2dx=\frac{2}{1+\kn^2}\left(1-\frac{E(\kn)}{K(\kn)}\right).
\]
\end{proposition}
\begin{proof}
By Lemma~\ref{AL3} we have
\begin{multline*}
\int_0^1u^{\pm}_n(x)^2dx
=\frac{2\kn^2}{1+\kn^2}\int_0^1\sn^2\left((2nx+1)K(\kn),\kn\right)dx\\
=\frac{2\kn^2}{1+\kn^2}\int_0^{2nK(\kn)}\sn^2\left(y+K(\kn),\kn\right)\frac{dy}{2nK(\kn)}
=\frac{2\kn^2}{2n(1+\kn^2)K(\kn)}\frac{K(\kn)-E(\kn)}{\kn^2}2n.
\end{multline*}
Then the conclusion holds.
\end{proof}

The following proposition will be used to calculate (\ref{T2F1}), (\ref{T2F2}) and (\ref{T2F3}).
\begin{proposition}\label{S2P4}
Let $\kn$ be given by (\ref{e}).
Then,
\begin{equation}\label{S2P4E0}
1-\kn^2=16\exp\left(-\frac{1}{\sqrt{2}n\e}\right)(1+o(1))\ \ \textrm{as}\ \ \e\to 0.
\end{equation}
\end{proposition}
\begin{proof}
The proof can be found in \cite[Proposition~3.1]{WY15}.
However, we show the proof for readers' convenience.

It follows from Lemma~\ref{AL2} that
\begin{equation}\label{S2P4E1-}
K(\kn)=2\log 2+\log\frac{1}{\sqrt{1-\kn^2}}+o(1)\ \ \textrm{as}\ \ \kn\to 1.
\end{equation}
Using (\ref{e}), we have
\[
\frac{1}{2n\e}=\sqrt{1+\kn^2}K(\kn)
=\sqrt{1+\kn^2}\left(2\log 2+\log\frac{1}{\sqrt{1-\kn^2}}+o(1)\right)\ \ \textrm{as}\ \ \kn\to 1,
\]
and hence
\begin{multline}\label{S2P4E1}
-2\log\frac{1}{\sqrt{1-\kn^2}}=4\log 2-\frac{1}{\sqrt{1+\kn^2}n\e}+o(1)\\
=4\log 2-\frac{1}{\sqrt{2}n\e}-\frac{1-\kn^2}{\sqrt{2(1+\kn^2)}(\sqrt{1+\kn^2}+\sqrt{2})n\e}+o(1)
\ \ \textrm{as}\ \ \kn\to 1.
\end{multline}
By \eqref{S2P4E1-} and (\ref{e}) we have
\begin{equation}\label{S2P4E2}
\frac{1-\kn^2}{\e}=2n\sqrt{1+\kn^2}(1-\kn^2)K(k_n)\to 0\ \ \textrm{as}\ \ \kn\to 1.
\end{equation}
The formula (\ref{S2P4E0}) follows from (\ref{S2P4E1}) and (\ref{S2P4E2}).
\end{proof}

\section{Critical eigenvalues and Hopf bifurcation}
Let $(u^{\pm}_n,0)$ be internal $n$-layer stationary solutions of (\ref{S1E1}) given by (\ref{un}), and let $\kn$ be a unique solution of (\ref{e}).
In this section we study the eigenvalue problem (\ref{EVP}).

The following cubic polynomial is important in the study of complex eigenvalues:
\begin{multline}\label{g}
g(\lambda,\tau):=
\frac{\tau}{\gamma}\lambda^3+\left(\frac{2\tau}{\gamma}+1\right)\lambda^2
+\left\{\frac{\alpha\beta}{\gamma}+2-\frac{3\tau}{\gamma}(1-\rho^2)^2\right\}\lambda\\
+\frac{3\alpha\beta}{\gamma}\left(1-\left<(u^{\pm}_n)^2,1\right>\right)
-3(1-\rho^2)^2.
\end{multline}
Let $\sigma_g$ be the roots of $g(\,\cdot\,,\tau)=0$, {\it i.e.},
\[
\sigma_g:=\{\lambda\in\C;\ g(\lambda,\tau)=0\}.
\]
Let
\[
\sigma_0:=\left\{-\frac{\gamma}{\tau}\right\}.
\]
\begin{lemma}\label{L1}
Let $n\ge 1$.
The following holds:
\begin{equation}\label{L1E0}
\sigma_{\calL}\subset \sigma_L\cup\sigma_g\cup\sigma_0.
\end{equation}
\end{lemma}
\begin{proof}
Let $\lambdaC\in\C\setminus(\sigma_L\cup\sigma_g\cup\sigma_0)$.
We show that, for every $(\PhiC,\XiC)\in Y$, the problem
\begin{equation}\label{L1E1}
\left\{\calL-(\lambdaC)\Id\right\}\left(
\begin{array}{c}
\phiC\\
\etaC
\end{array}
\right)=\left(
\begin{array}{c}
\PhiC\\
\XiC
\end{array}
\right)
\end{equation}
has a unique solution $(\phiC,\etaC)\in X$.
We write
\[
\lambda:=\lambdaC,\ \phi:=\phiC,\ \eta:=\etaC,\ \Phi:=\PhiC\ \textrm{and}\ \Xi:=\XiC.
\]
Since $\lambda\not\in\sigma_0$, by the second equation of (\ref{L1E1}) we have
\begin{equation}\label{L1E2}
\eta=\frac{\beta\left<\phi,1\right>-\tau\Xi}{\tau\lambda+\gamma}.
\end{equation}
Substituting (\ref{L1E2}) into the first equation of (\ref{L1E1}), we have
\begin{equation}\label{L1E3}
\left(L-\lambda\right)\phi-
\frac{\alpha\beta\left<\phi,1\right>}{\tau\lambda+\gamma}=
\Phi-\frac{\alpha\tau\Xi}{\tau\lambda+\gamma}.
\end{equation}
Let
\[
A\phi:=-\frac{\alpha\beta\left<\phi,1\right>}{\tau\lambda+\gamma}.
\]
Then, the linear operator defined by the LHS of (\ref{L1E3}) is $L+A-\lambda$.
Let $h(\lambda)$ be defined by (\ref{S2P5E1}).
Then, by Proposition~\ref{S2P2} we have
\begin{equation}\label{L1E3+}
h(\lambda)
=1-\frac{\alpha\beta\left<(L-\lambda)^{-1}[1],1\right>}{\tau\lambda+\gamma}
=1-\frac{\alpha\beta}{\tau\lambda+\gamma}\left<
\frac{-(3+\lambda)+3(u^{\pm}_n)^2}{\lambda^2+2\lambda-3(1-\rho^2)^2},1\right>.
\end{equation}
Since $\lambda\not\in \sigma_L\cup\sigma_0$, we see that $\lambda\not\in\{\mu^{(n)}_0,\mu^{(n)}_n,-\gamma/\tau\}$, and hence the denominator of (\ref{L1E3+}) does not vanish.
Then, $h(\lambda)=0$ if and only if 
\begin{equation}\label{L1E3++}
\left(\frac{\tau}{\gamma}\lambda+1\right)\left\{\lambda^2+2\lambda-3(1-\rho^2)^2\right\}+
\frac{\alpha\beta}{\gamma}\left\{\lambda+3\left(1-\left<(u^{\pm}_n)^2,1\right>\right)\right\}=0.
\end{equation}
This is a cubic equation for $\lambda$ and is equivalent to $g(\lambda,\tau)=0$.
Since we are assuming $\lambda\not\in\sigma_g$, we see that $h(\lambda)\neq 0$.
Because of Proposition~\ref{S2P5}~(i), we see that $L+A-\lambda$ has the inverse, and hence (\ref{L1E3}) can be solved for $\phi$, {\it i.e.,}
\begin{equation}\label{L1E4}
\phi=\left( L+A-\lambda\right)^{-1}\left[
\Phi-\frac{\alpha\tau\Xi}{\tau\lambda+\gamma}\right].
\end{equation}
Substituting (\ref{L1E4}) into (\ref{L1E2}), we have
\begin{equation}\label{L1E5}
\eta=\frac{\beta}{\tau\lambda+\gamma}
\left<\left( L+A-\lambda\right)^{-1}\left[
\Phi-\frac{\alpha\tau\Xi}{\tau\lambda+\gamma}
\right],1\right>
-\frac{\tau\Xi}{\tau\lambda+\gamma}.
\end{equation}
Thus, the pair (\ref{L1E4}) and (\ref{L1E5}) gives a solution for (\ref{L1E1}), and hence the inverse of $\calL-\lambda\Id$ exists.
Moreover, $\calL-\lambda\Id$ is injective and surjective.
It is obvious that the mapping from $(\phi,\eta)$ to $(\Phi,\Xi)$ is bounded.
It follows from the open mapping theorem that the inverse of $\calL-\lambda\Id$ is bounded.
Thus, $\lambda\not\in\sigma_{\calL}$ and (\ref{L1E0}) holds.
\end{proof}
Since $\sigma_L$ does not depend on $\tau$ and $-\gamma/\tau$ is always negative for $\tau>0$, it is important to study $\sigma_g$.

\begin{lemma}\label{L1+}
Let $n\ge 1$.
Assume that (\ref{T1E0}) or (\ref{T2E0}) holds.
Then the following (i)--(v) hold:\\
(i) The equation $g(\,\cdot\,,\tau)=0$ has a negative real root for all $\tau>0$, and hence $\sigma_g$ has a negative real number.\\
(ii) $\sigma_{\calL}$ has at most two complex eigenvalues.
Moreover, if they exist, then they are one pair of complex conjugate eigenvalues in $\sigma_g$.\\
(iii) $\sigma_{\calL}$ does not have a zero eigenvalue, {\it i.e.,} $0\not\in\sigma_{\calL}$.\\
(iv) If $\bigl<\psi^{(n)}_j,1\bigr>=0$, then $\mu^{(n)}_j\in\sigma_{\calL}$. In particular,
\[
\{\mu^{(n)}_1,\mu^{(n)}_3,\mu^{(n)}_5,\ldots\}\subset\sigma_{\calL}.
\]
(v) If $\bigl<\psi^{(n)}_j,1\bigr>\neq 0$, then $\mu^{(n)}_j\not\in\sigma_{\calL}$. In particular,
\[
\mu^{(n)}_0\not\in\sigma_{\calL}.
\]
\end{lemma}
\begin{proof} 
(i) Let $g(\lambda,\tau)$ be defined by (\ref{g}).
By Proposition~\ref{S2P3} and (\ref{T1E0}) (or (\ref{T2E0})) we have
\begin{multline}\label{L1+E0}
g(0,\tau)=
\frac{3\alpha\beta}{\gamma}\left(1-\left<(u^{\pm}_n)^2,1\right>\right)-3(1-\rho^2)^2\\
=\frac{3\alpha\beta}{\gamma}(1+\kn^2)\left\{2E(\kn)-(1-\kn^2)K(\kn)\right\}-3(1-\kn^2)^2K(\kn)>0.
\end{multline}
Note that the expression (\ref{delta}) comes from (\ref{L1+E0}).
Since $g(-\infty,\tau)=-\infty$, by the intermediate value theorem we see that $g(\lambda,\tau)=0$ has a negative real root.\\
(ii) Since $\sigma_L\cup\sigma_0\subset\R$, it follows from Lemma~\ref{L1} that the complex eigenvalues of $\calL$ come only from $\sigma_g$ if they exist.
The set $\sigma_g$ consists of the roots of the cubic equation $g(\lambda,\tau)=0$ with real coefficients.
Thus, the conclusion (ii) holds.\\
(iii) It follows from Proposition~\ref{S2P1}~(i) that $0\not\in\sigma_L$.
By (\ref{L1+E0}) we see that $g(0,\tau)>0$ and hence $0\not\in\sigma_g$.
By (\ref{L1E0}) we see that $0\not\in\sigma_{\calL}$.\\
(iv) We consider (\ref{EVP}) with $\lambdaC=\mu^{(n)}_j$.
Let $\phiC=\psi^{(n)}_j$ and $\etaC=0$.
Since $\bigl<\psi^{(n)}_j,1\bigr>=0$, we can easily check that $(\phiC,\etaC)$ is a nontrivial solution of (\ref{EVP}), and hence $\mu^{(n)}_j\in\sigma_{\calL}$.
Since $\psi^{(n)}_j$, $j=1,3,5,\ldots$, is anti-symmetric (Proposition~\ref{S2P1}~(ii)), we see $\bigl<\psi^{(n)}_j,1\bigr>=0$ for $j=1,3,5,\ldots$.
Therefore, $\{\mu^{(n)}_1,\mu^{(n)}_3,\mu^{(n)}_5,\ldots\}\subset\sigma_{\calL}$.\\
(v) Let $\phi:=\phiC$, $\eta:=\etaC$, $\Phi:=\PhiC$ and $\Xi:=\XiC$.
When $\lambdaC=\mu^{(n)}_j$, we show that (\ref{L1E1}) has a unique solution $(\phi,\eta)\in X$ for every $(\Phi,\Xi)\in Y$.
Let $\phi=c\psi^{(n)}_j+\psi^{\perp}$, $c\in\C$ and $\psi^{\perp}\in{\rm span}\{\psi^{(n)}_j\}^{\perp}$.
By the first equation of (\ref{L1E1}) we have
\begin{equation}\label{L1+E1}
(L-\mu^{(n)}_j)\psi^{\perp}=\alpha\eta+\Phi.
\end{equation}
Since $\mu^{(n)}_j$ is a simple eigenvalue of $L$, (\ref{L1+E1}) has a unique solution if and only if $\bigl<\alpha\eta+\Phi,\psi^{(n)}_j\bigr>=0$.
In this case we have
\begin{equation}\label{L1+E2}
\eta=-\frac{\bigl<\Phi,\psi^{(n)}_j\bigr>}{\alpha\bigl<\psi^{(n)}_j,1\bigr>},
\end{equation}
where we use $\bigl<1,\psi^{(n)}_j\bigr>=\bigl<\psi^{(n)}_j,1\bigr>\neq 0$.
By (\ref{L1+E1}) we have
\begin{equation}\label{L1+E3}
\psi^{\perp}=(L-\mu^{(n)}_j)^{-1}\left[\Phi-\frac{\bigl<\Phi,\psi^{(n)}_j\bigr>}{\bigl<\psi^{(n)}_j,1\bigr>}\right].
\end{equation}
By the second equation of (\ref{L1E1}) we have
\[
c=-\frac{\bigl<\psi^{\perp},1\bigr>}{\bigl<\psi^{(n)}_j,1\bigr>}+\frac{(\tau\mu^{(n)}_j+\gamma)\eta+\tau\Xi}{\beta\bigl<\psi^{(n)}_j,1\bigr>}.
\]
Since $\phi=c\psi^{(n)}_j+\psi^{\perp}$, we have
\begin{equation}\label{L1+E5}
\phi=\psi^{\perp}-\frac{\bigl<\psi^{\perp},1\bigr>}{\bigl<\psi^{(n)}_j,1\bigr>}\psi^{(n)}_j+
\frac{(\tau\mu^{(n)}_j+\gamma)\eta+\tau\Xi}{\beta\bigl<\psi^{(n)}_j,1\bigr>}\psi^{(n)}_j,
\end{equation}
where $\psi^{\perp}$ and $\eta$ are given by (\ref{L1+E3}) and (\ref{L1+E2}), respectively.
Thus, the pair (\ref{L1+E5}) and (\ref{L1+E2}) gives a unique solution for (\ref{L1E1}), and hence $\mu^{(n)}_j\not\in\sigma_{\calL}$.
In particular, $\psi^{(n)}_0$ does not change sign, and hence $\bigl<\psi^{(n)}_0,1\bigr>\neq 0$.
Thus, $\mu^{(n)}_0\not\in\sigma_{\calL}$.
\end{proof}

\begin{lemma}\label{L2}
Let $n\ge 1$.
Assume that (\ref{T1E0}) holds if $n=1$ and that (\ref{T2E0}) holds if $n\ge 2$.
Then the following (i)--(iii) hold:\\
(i) Let $\lIn$ be given by (\ref{T1E1}).
Then, there exists a unique $\tau_n>0$ which is defined by (\ref{T1E-1}) such that if $\tau=\tau_n$, then $\sigma_{\calL}$ has one pair of complex conjugate eigenvalues on the imaginary axis which are denoted by $\pm i\lIn$, $\lIn>0$, and that the other eigenvalues of $\sigma_{\calL}$ are real.\\
(ii) If $|\tau-\tau_n|$ is small, then $\calL$ has exactly one pair of complex eigenvalues $\lR(\tau)\pm i\lI(\tau)$ such that $\lR(\tau_n)=0$, $\lI(\tau_n)=\lambda_{{\rm I},n}$ and both $\lR(\tau)$ and $\lI(\tau)$ are $C^1$-functions of $\tau$.\\
(iii) $\frac{d\lR}{d\tau}(\tau_n)>0$.
\end{lemma}
\begin{proof}
(i) It is enough to look for a pure imaginary root $i\lI$, $\lI>0$, of the equation $g(i\lI,\tau)=0$, because $-i\lI$ becomes the other complex root.
Let $\delta:=\alpha\beta/\gamma$.
We have
\begin{multline*}
g(i\lI,\tau)=
-i\frac{\tau}{\gamma}\lI^3-\left(2\frac{\tau}{\gamma}+1\right)\lI^2+i\left\{\delta+2-3\frac{\tau}{\gamma}(1-\rho^2)^2\right\}\lI
\\
+3\delta\left(1-\left<(u^{\pm}_n)^2,1\right>\right)-3(1-\rho^2)^2=0.
\end{multline*}
Taking real and imaginary parts, we have
\begin{equation}\label{L2E1}
\begin{cases}
\disp -\left(2\frac{\tau}{\gamma}+1\right)\lI^2+3\delta\left(1-\left<(u^{\pm}_n)^2,1\right>\right)-3(1-\rho^2)^2=0,\\
\disp -\frac{\tau}{\gamma}\lI^3+\left\{\delta+2-3\frac{\tau}{\gamma}(1-\rho^2)^2\right\}\lI=0.
\end{cases}
\end{equation}
Since $\lI\neq 0$, by the second equation of (\ref{L2E1}) we have
\begin{equation}\label{L2E2}
\tau=\frac{(\delta+2)\gamma}{\lI^2+3(1-\rho^2)^2}.
\end{equation}
Substituting (\ref{L2E2}) into the first equation of (\ref{L2E1}), we have
\begin{equation}\label{L2E3}
Z^2+\left\{2(\delta+2)-3\delta\left(1-\left<(u^{\pm}_n)^2,1\right>\right)\right\}Z-6(\delta+2)(1-\rho^2)^2=0,
\end{equation}
where
\begin{equation}\label{L2E4}
Z:=\lI^2+3(1-\rho^2)^2.
\end{equation}
We can easily see that the equation (\ref{L2E3}) for $Z$ has a positive root $Z_+$ and a negative root $Z_-$.
Because of (\ref{L2E4}), $Z_-$ does not give a real root $\lI$.
We consider the case $Z=Z_+$.
By (\ref{L2E3}) we have
\[
Z_+=\frac{1}{2}\left[
-\left\{
2(\delta+2)-3\delta\left(1-\left<(u^{\pm}_n)^2,1\right>\right)
\right\}+\sqrt{D}
\right],
\]
where
\begin{equation}\label{D}
D:=\left\{2(\delta+2)-3\delta\left(1-\left<(u^{\pm}_n)^2,1\right>\right)\right\}^2+24(\delta+2)(1-\rho^2)^2.
\end{equation}
By (\ref{L2E4}) we have
\begin{equation}\label{L2E5}
\lI^2=-3(1-\rho^2)^2+\frac{1}{2}
\left[-\left\{2(\delta+2)-3\delta\left(1-\left<(u^{\pm}_n)^2,1\right>\right)\right\}+\sqrt{D}\right].
\end{equation}
Using (\ref{T1E0}) (or (\ref{T2E0})), we have
\begin{align*}
(\sqrt{D})^2-&\left\{2(\delta+2)-3\delta\left(1-\left<(u^{\pm}_n)^2,1\right>\right)+6(1-\rho^2)^2\right\}^2\\
&=36(1-\rho^2)^2\left\{\delta\left(1-\left<(u^{\pm}_n)^2,1\right>\right)-(1-\rho^2)^2\right\}\\
&=36(1-\rho^2)^2\left[
\delta(1+\kn^2)\left\{2E(\kn)-(1-\kn^2)K(\kn)\right\}-(1-\kn^2)^2K(\kn)\right]>0.
\end{align*}
The RHS of (\ref{L2E5}) is positive.
Hence, (\ref{L2E5}) gives a positive root $\lambda_{{\rm I},n}$ which is a unique positive root of the equation (\ref{L2E3}) for $\lI$.
By (\ref{L2E5}) we have
\begin{align}
\lambda_{{\rm I},n} &=\frac{1}{\sqrt{2}}\left[
\sqrt{D}-\left\{
2(\delta+2)-3\delta\left(1-\left<(u^{\pm}_n)^2,1\right>\right)+6(1-\rho^2)^2\right\}\right]^{1/2}\label{L2E6}\\
&=\frac{1}{\sqrt{2}}\left[
\sqrt{D}-\left\{
2(\delta+2)-\frac{3\delta}{1+\kn^2}\left(2\frac{E(\kn)}{K(\kn)}-1+\kn^2\right)+6\left(\frac{1-\kn^2}{1+\kn^2}\right)^2\right\}\right]^{1/2},\nonumber
\end{align}
where
\[
D=\left\{
2(\delta+2)-\frac{3\delta}{1+\kn^2}\left(2\frac{E(\kn)}{K(\kn)}-1+\kn^2\right)
\right\}^2
+24(\delta+2)\left(\frac{1-\kn^2}{1+\kn^2}\right)^2.
\]
We obtain (\ref{T1E1}).
Substituting (\ref{L2E5}) into (\ref{L2E2}), we obtain $\tau_n>0$ such that $i\lI$ is a root of $g(\,\cdot\,,\tau_n)=0$.
Specifically, we have
\begin{align*}
\tau_n:
&=\frac{2(\delta+2)\gamma}{\sqrt{D}-\left\{2(\delta+2)-3\delta\left(1-\left<(u^{\pm}_n)^2,1\right>\right)\right\}}\\
&=\frac{2(\delta+2)\gamma}{\sqrt{D}-2(\delta+2)+\frac{3\delta}{1+\kn^2}\left(2\frac{E(\kn)}{K(\kn)}-1+\kn^2\right)}.
\end{align*}
We obtain (\ref{T1E-1}).
In summary, there exists a unique $\tau_n>0$ such that $\sigma_{\calL}$ has pure imaginary eigenvalues $\pm i\lambda_{{\rm I},n}$ and that the other eigenvalues are real (Lemma~\ref{L1}~(ii)).
The proof of (i) is complete.\\
(ii) In the proof we consider $\tau$ as a complex variable and consider $g(\lambda,\tau)$ as a complex analytic function of two complex variables $\lambda$ and $\tau$.
We have
\[
\frac{\p g}{\p\lambda}(i\lIn,\tau_n)
=-2(\delta+2)+i2\left(2\frac{\tau_n}{\gamma}+1\right)\lIn\neq 0.
\]
Since $g(\lambda,\tau)$ is a complex analytic function, it follows from the implicit function theorem that there exists a complex analytic function $\lambda=\lambda(\tau)$, which is also denoted by $\lR(\tau)+i\lI(\tau)$, defined in a neighborhood of $\tau_n$ such that $\lambda(\tau_n)=i\lIn$ and that $g(\lambda(\tau),\tau)=0$ for $\tau$ close to $\tau_n$.
Therefore, if $\tau$ is close to $\tau_n$, then $\sigma_{\calL}$ has exactly one pair of complex conjugate eigenvalues which are close to $\pm i\lIn$.
Thus, the proof of (ii) is complete.\\
(iii) By direct calculation we have
\[
\frac{\p\lambda}{\p\tau}(\tau_n)
=-\frac{\frac{\p g}{\p\tau}(i\lIn,\tau_n)}{\frac{\p g}{\p\lambda}(i\lIn,\tau_n)}
=\frac{\zeta_{\rm R}+i\zeta_{\rm I}}{\Delta},
\]
where
\begin{align*}
\Delta&:=4(\delta+2)^2\gamma+\lIn^2\left(4\frac{\tau_n}{\gamma}+2\right)^2\gamma>0,\\
\zeta_{\rm R}&:=2\lIn^2\left\{\lIn^2+3(1-\rho^2)^2\right\}>0,\\
\zeta_{\rm I}&:=-2(\delta+2)\lIn\left\{\lIn^2+3(1-\rho^2)^2\right\}-2\lIn^3\left(4\frac{\tau_n}{\gamma}+2\right)<0.
\end{align*}
Thus,
\[
\frac{d\lR}{d\tau}(\tau_n)={\rm Re}\frac{\p\lambda}{\p\tau}(\tau_n)=\frac{\zeta_{\rm R}}{\Delta}>0.
\]
and
\begin{equation}\label{L2E7}
\frac{d\lI}{d\tau}(\tau_n)={\rm Im}\frac{\p\lambda}{\p\tau}(\tau_n)=\frac{\zeta_{\rm I}}{\Delta}<0.
\end{equation}
The proof of (iii) is complete.
\end{proof}

\begin{lemma}\label{L2+}
Let $n\ge 1$.
Assume that (\ref{T1E0}) holds if $n=1$ and that (\ref{T2E0}) holds if $n\ge 2$.
The two complex eigenvalues $\lR(\tau)\pm i\lI(\tau)$ obtained in Lemma~\ref{L2} are simple if $|\tau-\tau_n|$ is small.
\end{lemma}
\begin{proof}
It follows from Lemma~\ref{L2}~(i) and (ii) that $\lR(\tau)\pm i\lI(\tau)$ are continuous in $\tau$ and $\lR(\tau_n)+i\lI(\tau_n)=i\lIn$.
Hereafter, we show that $\lambda(\tau):=\lR(\tau)+i\lI(\tau)$ is simple, since a proof for the case $\lR(\tau)-i\lI(\tau)$ is similar.

Let $(\phi(\tau),\eta(\tau))$ be an eigenvector of (\ref{EVP}) associated with $\lambda(\tau)$.
Solving the second equation of (\ref{EVP}) with respect to $\eta$, we have
\begin{equation}\label{L2+EE1}
\eta(\tau)=\frac{\beta\left<\phi(\tau),1\right>}{\tau\lambda(\tau)+\gamma},
\end{equation}
where $\tau\lambda(\tau)+\gamma\neq 0$.
Substituting (\ref{L2+EE1}) into the first equation of (\ref{EVP}), we have
\[
(L-\lambda(\tau))\phi-\frac{\alpha\beta\left<\phi(\tau),1\right>}{\tau\lambda(\tau)+\gamma}=0.
\]
Since $\lambda(\tau)\in \sigma_{\calL}\setminus\sigma_L$, by Proposition~\ref{S2P5} we have $h(\lambda)=0$ and
\begin{equation}\label{L2+EE2}
\phi(\tau)=(L-\lambda(\tau))^{-1}[1]
=\frac{-(3+\lambda(\tau))+3(u^{\pm}_n)^2}{\lambda(\tau)^2+2\lambda(\tau)-3(1-\rho^2)^2},
\end{equation}
where Proposition~\ref{S2P2} was used.
By the expressions (\ref{L2+EE2}) and (\ref{L2+EE1}) we see that the geometric multiplicity of $\lambda(\tau)$ is $1$ and that $(\phi(\tau),\eta(\tau))$ continuously depends on $\tau$.

We can easily check that the adjoint operator of $\calL$ is
\[
\calL^*=\left(
\begin{array}{cc}
L & \frac{\beta}{\tau}\\
-\alpha\left<\,\cdot\,,1\right> & -\frac{\gamma}{\tau}
\end{array}
\right).
\]
Since $\lambda(\tau)$ is an eigenvalue of $\calL$, $\overline{\lambda(\tau)}$ is an eigenvalue of $\calL^*$ with the same geometric and algebraic multiplicities as the eigenvalue $\lambda(\tau)$ of $\calL$.
We consider the eigenvalue problem
\[
\left(\calL^*-\overline{\lambda(\tau)}{\rm Id}\right)
\left(\begin{array}{c}\phi^*(\tau)\\\eta^*(\tau)\end{array}\right)=0,
\]
where $(\phi^*(\tau),\eta^*(\tau))$ is an associated eigenvector.
By the same procedure as above we have
\[
\phi^*(\tau)=
\left(L-\overline{\lambda(\tau)}\right)^{-1}[1]
=\frac{-\left(3+\overline{\lambda(\tau)}\right)+3(u^{\pm}_n)^2}
{\overline{\lambda(\tau)}^2+2\overline{\lambda(\tau)}-3(1-\rho^2)^2}
\ \ \textrm{and}\ \ 
\eta^*(\tau)=\frac{-\alpha\left<\phi^*(\tau),1\right>}{\tau\overline{\lambda(\tau)}+\gamma}.
\]
Thus, $(\phi^*(\tau),\eta^*(\tau))$ also continuously depends on $\tau$.
We show by contradiction that the algebraic multiplicity of $\lambda (\tau)$ is $1$.
Suppose the contrary, {\it i.e.,} there exists
\[
(\psi(\tau),\zeta(\tau))\in{\rm Ker}(\calL-\lambda(\tau){\rm Id})^2\setminus{\rm Ker}(\calL-\lambda(\tau){\rm Id})
\]
such that
\[
\left(\begin{array}{c}\phi(\tau) \\\eta(\tau)\end{array}\right)
=\left(\calL-\lambda(\tau){\rm Id}\right)
\left(\begin{array}{c}\psi(\tau) \\\zeta(\tau)\end{array}\right).
\]
Let
\[
J(\tau):=
\left<
\left(
\begin{array}{c}
\phi(\tau)\\
\eta(\tau)
\end{array}
\right),
\left(
\begin{array}{c}
\phi^*(\tau)\\
\eta^*(\tau)
\end{array}
\right)
\right>=
\int_0^1\phi(\tau)\overline{\phi^*(\tau)}dx+\eta(\tau)\overline{\eta^*(\tau)}.
\]
Then,
\begin{align*}
J(\tau)
&=\left<\left(\calL-\lambda(\tau){\rm Id}\right)
\left(\begin{array}{c}\psi(\tau) \\\zeta(\tau)\end{array}\right),
\left(\begin{array}{c}\phi^*(\tau) \\\eta^*(\tau)\end{array}\right) \right>\\
&=\left<
\left(\begin{array}{c}\psi(\tau) \\\zeta(\tau)\end{array}\right),
\left(\calL^*-\overline{\lambda(\tau)}{\rm Id}\right)
\left(\begin{array}{c}\phi^*(\tau) \\\eta^*(\tau)\end{array}\right)
\right>=0.
\end{align*}
On the other hand, we show later that $J(\tau)\neq 0$ if $|\tau-\tau_n|$ is small.
This is a contradiction, and hence the algebraic multiplicity of $\lambda(\tau)$ is $1$.

It is enough to show that
\begin{equation}\label{L2+EE3}
J(\tau_n)\neq 0,
\end{equation}
since $J(\tau)$ is continuous in $\tau$.
Let $\phi:=\phi(\tau_n)$ and $\phi^*:=\phi^*(\tau_n)$ for simplicity.
Since $h(i\lIn)=0$, we have
\begin{equation}\label{L2+E6}
\frac{\alpha\beta\left<\phi,1\right>}{i\tau\lIn+\gamma}=1.
\end{equation}
Since $\overline{\phi^*}=\phi$, by (\ref{L2+E6}) we have
\begin{multline*}
\int_0^1\phi\overline{\phi^*}dx+\eta\overline{\eta^*}
=\int_0^1\phi^2dx+\frac{\beta\left<\phi,1\right>}{i\tau\lIn+\gamma}
\overline{\frac{-\alpha\tau\left<\phi^*,1\right>}{-i\tau\lIn+\gamma}}\\
=\int_0^1\phi^2dx+\frac{\beta\left<\phi,1\right>}{i\tau\lIn+\gamma}\frac{-\alpha\tau\left<\phi,1\right>}{i\tau\lIn+\gamma}
=\int_0^1\phi^2dx-\frac{\tau}{\alpha\beta}.
\end{multline*}
By (\ref{L2+EE2}) we have
\begin{equation}\label{L2+E7+}
\phi=(L-i\lIn)^{-1}[1]=\frac{-(3+i\lIn)+3(u^{\pm}_n)^2}{-\lIn^2+2i\lIn-3(1-\rho^2)^2}.
\end{equation}
Using (\ref{L2+E7+}) and (\ref{L2E2}), we have
\[
\int_0^1\phi^2dx-\frac{\tau}{\alpha\beta}
=\frac{\int_0^1\left\{(-3+3(u^{\pm}_n)^2)-i\lIn\right\}^2dx}{\left\{-\lIn^2+2i\lIn-3(1-\rho^2)^2\right\}^2}
-\frac{(\delta+2)\gamma}{\alpha\beta\left\{\lIn^2+3(1-\rho^2)^2\right\}}.
\]
It is enough to show that
\begin{equation}\label{L2+E8}
\int_0^1\left\{(-3+3(u^{\pm}_n)^2)-i\lIn\right\}^2dx
-\frac{(\delta+2)\gamma\left\{-\lIn^2+2i\lIn-3(1-\rho^2)^2\right\}^2}{\alpha\beta\left\{\lIn^2+3(1-\rho^2)^2\right\}}
\neq 0.
\end{equation}
Taking the imaginary part of the  LHS of (\ref{L2+E8}), we have
\[
6\lIn\left(1-\left<(u^{\pm}_n)^2,1\right>\right)+4\lIn\frac{(\delta+2)\gamma}{\alpha\beta}>0.
\]
Thus, (\ref{L2+EE3}) holds, and hence the conclusion of the lemma holds.
\end{proof}

\begin{lemma}\label{L3}
Let $n\ge 1$.
Assume that (\ref{T1E0}) holds if $n=1$ and that (\ref{T2E0}) holds if $n\ge 2$.
Let $\tau_n>0$ be given in Lemma~\ref{L2}~(i).
Then the following (i)--(iii) hold:\\
(i) If $n\ge 2$ and $\tau=\tau_n$, then
\[
\sigma_{\calL}\cap \calH_+\subset\{ \mu^{(n)}_1,\mu^{(n)}_2,\ldots,\mu^{(n)}_{n-1}\}
\ \ \textrm{and}\ \ \sigma_{\calL}\cap i\R=\{i\lambda_{{\rm I},n},-i\lambda_{{\rm I},n}\}.
\]
If $n=1$ and $\tau=\tau_1$, then
\[
\sigma_{\calL}\cap \calH_+=\emptyset\ \ \textrm{and}\ \ \sigma_{\calL}\cap i\R=\{i\lambda_{{\rm I},n},-i\lambda_{{\rm I},n}\}.
\]
(ii) If $n\ge 2$ and $0<\tau<\tau_n$, then
\[
\sigma_{\calL}\cap\left( i\R\cup \calH_+\right)\subset\{\mu^{(n)}_1,\mu^{(n)}_2,\ldots,\mu^{(n)}_{n-1}\}.
\]
If $n=1$ and $0<\tau<\tau_1$, then
\[
\sigma_{\calL}\cap(i\R\cup \calH_+)=\emptyset.
\]
(iii) If $n\ge 1$ and $\tau>\tau_n$, then $\sigma_{\calL}\cap \calH_+\neq\emptyset$.
\end{lemma}
\begin{proof} 
(i) Let $\tau=\tau_n$.
It follows from Lemma~\ref{L1+}~(ii) and Lemma~\ref{L2}~(i) that two pure imaginary eigenvalues $\pm i\lambda_{\rm{I},n}$ exist and the other eigenvalues are real.
Because of Lemmas~\ref{L1} and \ref{L1+}~(i) and (v),
we see that $\sigma_{\calL}\cap\calH_+\subset\{\mu^{(n)}_1,\ldots,\mu^{(n)}_{n-1}\}$ for $n\ge 2$ and that $\sigma_{\calL}\cap\calH_+=\emptyset$ for $n=1$.\\
(ii) It follows from Lemma~\ref{L1+}~(ii) and Lemma~\ref{L2}~(ii) that if $\tau$ is close to $\tau_n$, then $\sigma_{\calL}$ has exactly one pair of complex conjugate eigenvalues which are denoted by $\lR(\tau)\pm i\lI(\tau)$ and they are near $i\R$.
Since $\lR(\tau)\pm i\lI(\tau)$ are roots of $g(\lambda,\tau)=0$, they depend continuously on $\tau$.
Because of the uniqueness of $\tau_n$ and Lemma~\ref{L1+}~(iii), the eigenvalues $\lR(\tau)\pm i\lI(\tau)$ cannot be on $i\R$ for $\tau<\tau_n$, and hence $\lR(\tau)\pm i\lI(\tau)\in\calH_-$ for $\tau<\tau_n$.
Thus the conclusion holds.\\
(iii) By a similar argument as in (ii) we see that $\lR(\tau)\pm i\lI(\tau)\in\calH_+$ for $\tau>\tau_n$, since $\lR(\tau)\pm i\lI(\tau)\in\calH_+$ cannot be on $i\R$ for $\tau>\tau_n$.
Thus the conclusion holds.
\end{proof}

\section{Asymptotic formulas}
\begin{lemma}
The following asymptotic formulas hold: As $\e\to 0$,
\begin{align}
\lIn &=24\sqrt{\frac{\delta}{\delta+2}}\left({\sqrt{2}n\e}\right)^{1/2}\exp\left(-\frac{1}{\sqrt{2}n\e}\right)(1+o(1)),\label{L4E1}\\
\tau_n &=\frac{\gamma(\delta+2)}{192}\exp\left(\frac{\sqrt{2}}{n\e}\right)(1+o(1)),\label{L4E2}\\
\chi_n&=\frac{16\sqrt{2}}{n\e}\exp\left(-\frac{\sqrt{2}}{n\e}\right)(1+o(1)).\label{L4E3}
\end{align}
\end{lemma}
\begin{proof}
We derive asymptotic formulas as $\e\to 0$.
Then, $\kn\to 1$ as $\e\to 0$.
First, we derive an asymptotic formula of $E(\kn)/K(\kn)$ as $\kn\to 1$ to compare $E(\kn)/K(\kn)$ and $1-\kn^2$.
By Lemmas~\ref{AL2} and \ref{S2P4} we have
\begin{multline}\label{L4PE-1}
K(\kn)=2\log 2-\frac{1}{2}\log(1-\kn^2)+o(1)\\
=2\log 2-\frac{1}{2}\log\left(16\exp\left(-\frac{1}{\sqrt{2}n\e}\right)(1+o(1))\right)+o(1)
=\frac{1}{2\sqrt{2}n\e}+o(1).
\end{multline}
Since $E(\kn)=1+o(1)$, we have
\begin{equation}\label{L4PE-1+1}
\frac{E(\kn)}{K(\kn)}=2\sqrt{2}n\e +o(1).
\end{equation}
Since $1-\kn^2=16\exp(-1/\sqrt{2}n\e)(1+o(1))$ (Proposition~\ref{S2P4}), we see that
\begin{equation}\label{L4PE0}
1-\kn^2=o\left(\frac{E(\kn)}{K(\kn)}\right).
\end{equation}
In particular, $(1-\kn^2)K(\kn)\to 0$ as $\kn\to 0$.

Next, we prove (\ref{L4E1}), (\ref{L4E2}) and (\ref{L4E3}).
By Proposition~\ref{S2P3} we have
\begin{equation}\label{L4PE1}
1-\left<(u^{\pm}_n)^2,1\right>=
1-\frac{2}{1+\kn^2}\left(1-\frac{E(\kn)}{K(\kn)}\right)
=\frac{E(\kn)}{K(\kn)}(1+o(1)).
\end{equation}
Let $D$ be defined by (\ref{D}). Then by (\ref{L4PE1}) we have
\begin{multline}\label{L4PE3}
(\sqrt{D})^2-\left\{2(\delta+2)-3\delta\left(1-\left<(u^{\pm}_n)^2,1\right>\right)\right\}^2\\
=24(\delta+2)(1-\rho^2)^2
=24(\delta+2)\left(\frac{1-\kn^2}{1+\kn^2}\right)^2.
\end{multline}
On the other hand, by (\ref{L4PE1}) we have
\begin{align}
\sqrt{D}+&\left\{2\left(\delta+2)-3\delta(1-\left<(u^{\pm}_n)^2,1\right>\right)\right\}\nonumber\\
&=\left\{2(\delta+2)-3\delta\left(1-\left<(u^{\pm}_n)^2,1\right>\right)\right\}
\left[1+\frac{24(\delta+2)(1-\rho^2)^2}
{\left\{2(\delta+2)-3\delta\left(1-\left<(u^{\pm}_n)^2,1\right>\right)\right\}^2}\right]^{1/2}\nonumber\\
&\qquad+2(\delta+2)-3\delta\left(1-\left<(u^{\pm}_n)^2,1\right>\right)\nonumber\\
&=\left\{2(\delta+2)-3\delta\left(1-\left<(u^{\pm}_n)^2,1\right>\right)\right\}
(2+O((1-\rho^2)^2))\nonumber\\
&=4(\delta+2)-6\delta\frac{E(\kn)}{K(\kn)}+O((1-\kn^2)^2).\label{L4PE4}
\end{align}
By (\ref{L4PE4}) and (\ref{L4PE3}) we have
\begin{align}
\sqrt{D}-&\left\{2(\delta+2)-3\delta\left(1-\left<(u^{\pm}_n)^2,1\right>\right)\right\}
=\frac{(\sqrt{D})^2-\left\{2(\delta+2)-3\delta\left(1-\left<(u^{\pm}_n)^2,1\right>\right)\right\}^2}{\sqrt{D}+\left\{2(\delta+2)-3\delta\left(1-\left<(u^{\pm}_n)^2,1\right>\right)\right\}}\nonumber\\
&=\frac{24(\delta+2)\left(\frac{1-\kn^2}{1+\kn^2}\right)^2}{4(\delta+2)-6\delta\frac{E(\kn)}{K(\kn)}+O((1-k^2)^2)}\nonumber\\
&=6\left(\frac{1-\kn^2}{1+\kn^2}\right)^2\left(
1+\frac{3\delta}{2(\delta+2)}\frac{E(\kn)}{K(\kn)}+o\left(\frac{E(\kn)}{K(\kn)}\right)+O((1-\kn^2)^2)
\right).\label{L4PE5}
\end{align}
Since $\lIn^2$ is given by (\ref{L2E6}), by (\ref{L4PE5}) and (\ref{L4PE0}) we have
\begin{align*}
\lIn^2
&=\frac{1}{2}\left[
\sqrt{D}-\left\{2(\delta+2)-3\delta\left(1-\left<(u^{\pm}_n)^2,1\right>\right)\right\}-6(1-\rho^2)^2
\right]\\
&=\frac{1}{2}6\left(\frac{1-\kn^2}{1+\kn^2}\right)^2\frac{3\delta}{2(\delta+2)}\frac{E(\kn)}{K(\kn)}(1+o(1)).
\end{align*}
Using (\ref{L4PE-1+1}), Proposition~\ref{S2P4} and the relation $1+\kn^2=2+o(1)$, we obtain (\ref{L4E1}).
By (\ref{L4PE5}) and Proposition~\ref{S2P4} we have
\[
\tau=\frac{2(\delta+2)\gamma}{\sqrt{D}-\left\{2(\delta+2)-3\delta\left(1-\left<(u^{\pm}_n)^2\right>\right)\right\}}=\frac{(\delta+2)\gamma}{3}\left(\frac{1+\kn^2}{1-\kn^2}\right)^2(1+o(1)).
\]
Then, we obtain (\ref{L4E2}).
Let $\chi_n$ be given by (\ref{T1E0}).
Using (\ref{L4PE-1}), Proposition~\ref{S2P4} and the relation $E(\kn)=1+o(1)$, we have
\[
\chi_n=\frac{(1-\kn^2)^2K(\kn)}{(1+\kn^2)\left\{2E(\kn)-(1-\kn^2)K(\kn)\right\}}
=\frac{256\exp\left(-\frac{\sqrt{2}}{n\e}\right)\frac{1}{2\sqrt{2}n\e}}{(2+o(1))(2+o(1))}.
\]
Then, we obtain (\ref{L4E3}).
The proof is complete.
\end{proof}

\section{Proofs of Theorems \ref{T1} and \ref{T2}}
We consider the abstract evolution equation with a parameter $\tau$ in a Banach space $\calZ$
\begin{equation}\label{S5E1}
\frac{dx}{dt}=F(x,\tau).
\end{equation}
We will use the Hopf bifurcation theorem of Crandall-Rabinowitz~\cite{CR77}.
The following formulation is due to \cite[Theorem I.8.2]{K04}.
\begin{theorem}\label{S5T1}
Suppose that $F(x,\tau)$ satisfies the following (\ref{S5T1E1})--(\ref{S5T1E4}):
\begin{align}
& \textrm{$F:\calU\times\calV\to\calZ$, where $0\in\calU\subset\calX$ and $\tau_c\in\calV\subset\R$ are open neighborhoods,}\label{S5T1E1}\\
& \textrm{$F(0,\tau_c)=0$ and $D_xF(0,\tau)$ exists in $L(\calX,\calZ)$ for all $\tau\in\calV$},\label{S5T1E2}\\
& \textrm{$\calX\subset\calZ$ is continuously embedded},\label{S5T1E3}\\
& \textrm{$F\in C^3(\calU\times\calV,\calZ)$}.\label{S5T1E4}
\end{align}
Suppose that the following (\ref{S5T1E5})--(\ref{S5T1E7}) hold:
\begin{align}\label{S5T1E5} 
\begin{aligned}
&\textrm{$A_0:=D_xF(0,\tau_c)$ as a mapping in $\calZ$ with dense domain $D(A_0)=:\calX$ generates}\\
&\textrm{an analytic semigroup $e^{A_0t}\in L(\calZ,\calZ)$, $t\ge 0$, that is compact for $t>0$,}
\end{aligned}
\end{align}
\begin{align}\label{S5T1E6}
\begin{aligned}
& \textrm{$D_xF(0,\tau)\vphi(\tau)=\lambda(\tau)\vphi(\tau)$ with $\lambda(\tau_c)=i\kappa_0\in i\R\setminus\{0\}$, $\lambda(\tau)$ are simple eigenvalues}\\
& \textrm{ and ${\rm Re}\frac{d\lambda}{d\tau}(\tau_c)\neq 0$ (transversality condition),}
\end{aligned}
\end{align}
\begin{align}\label{S5T1E7}
& \textrm{$in\kappa_0$ is not an eigenvalue of $A_0$ for all $n\in\mathbb{Z}\setminus\{1,-1\}$.}
\end{align}
Then, there exist $\rho_0>0$, a continuous function $\kappa(\rho)$, $|\rho|<\rho_0$, and a continuous differentiable curve $(x(\rho),\tau(\rho))$ of $2\pi/\kappa(\rho)$-periodic solutions of (\ref{S5E1}) through $(x(0),\tau(0))=(0,\tau_c)$ with $2\pi/\kappa(0)=2\pi/\kappa_0$.
Every other periodic solution of (\ref{S5E1}) in a neighborhood of $(0,\tau_c)$ is obtained by a phase shift $S_{\theta}x(\rho)$.
In particular, $x(-\rho)=S_{\pi/\kappa(\rho)}x(\rho)$, $\kappa(-\rho)=\kappa(\rho)$ and $\tau(-\rho)=\tau(\rho)$ for all $|\rho|<\rho_0$.
\end{theorem}

We briefly see that our problem (\ref{S1E1}) satisfies (\ref{S5T1E1})--(\ref{S5T1E5}).
Let
\[
\calX:=H^2_N(I)\times\C\ \ \textrm{and}\ \ \calZ:=L^2(I)\times\C.
\]
Then, (\ref{S5T1E3}) is satisfied.
Let $(\bu,\bxi)$ be a steady state of (\ref{S1E1}), and let $u(x,t):=\bu(x)+U(x,t)$ and $\xi(t)=\bxi+\Xi(t)$.
Then, (\ref{S1E1}) can be written as follows:
\[
\frac{\partial}{\partial t}\left(
\begin{array}{c}
U\\
\Xi
\end{array}
\right)=\calL_0\left(
\begin{array}{c}
U\\
\Xi
\end{array}
\right)+\calF(U,\Xi),
\]
where
\[
\calL_0\left(
\begin{array}{c}
U\\
\Xi
\end{array}
\right)=\left(
\begin{array}{c}
\left(\e^2\frac{\partial^2}{\partial x^2}+f'(\bu)\right) U-\alpha\Xi\\
\frac{\beta}{\tau}\left<U,1\right>-\frac{\gamma}{\tau}\Xi
\end{array}
\right),
\]
\[
\calF(U,\Xi):=
\left(
\begin{array}{c}
f(\bu+U)-f(\bu)-f'(\bu)U\\
0
\end{array}
\right).
\]
We define\[
F:=\calL_0+\calF.
\]
Because of the continuous embedding $H^2_N(I)\to C(I)$, there exist small neighborhoods $\calU\subset\calX$ and $\calV\subset\R$ such that $F:\calU\times\calV\to\calZ$ and $F\in C^3(\calU\times\calV,\calZ)$.
Therefore, (\ref{S5T1E1}) and (\ref{S5T1E4}) hold.
Since $\calF(0,\tau_c)=0$ and $\calL_0\in L(\calX,\calZ)$, the condition (\ref{S5T1E2}) is satisfied.

\begin{proposition}\label{S5P1}
Let $\calX$, $\calZ$ and $\calL_0$ be defined as above.
The operator $\calL_0$, whose domain is $\calX$, is a sectorial operator on $\calZ$, and hence it generates a strongly continuous analytic semigroup on $\calZ$.
Moreover, there are $\theta\in (\frac{\pi}{2},\pi)$ and $a>0$ such that
\[
S_{a,\theta}:=\left\{
z\in\C;\ |\arg(z-a)|<\theta\right\}\cup\{a\}
\]
is in the resolvent set of $\calL_0$, that, for $\lambda\in S_{a,\theta}$, the operator $\calR(\lambda,\calL_0):=(\lambda-\calL_0)^{-1}$ is compact as an operator of $L(\calZ,\calZ)$ and that there exists $M>0$ such that
\begin{equation}\label{S5P1E1}
\left\|\calR(\lambda,\calL_0)\right\|\le\frac{M}{|\lambda-a|}\ \ \textrm{for}\ \ \lambda\in S_{a,\theta}.
\end{equation}
\end{proposition}
It follows from Lemmas~\ref{L1} and \ref{L1+} that $S_{a,\theta}$ is in the resolvent set of $\calL_0$. 
A proof of (\ref{S5P1E1}) is almost the same as that of \cite[(4.25)]{GMW21}.
Another proof of (\ref{S5P1E1}) can be found in \cite[Section 5]{KEM92}.
Other assertions follow from a theory of an analytic semigroup.
We omit the proof.
The condition (\ref{S5T1E5}) follows from Proposition~\ref{S5P1}.\\

In proofs of Theorems~\ref{T1} and \ref{T2} we check (\ref{S5T1E6}) and (\ref{S5T1E7}) to apply Theorem~\ref{S5T1} to (\ref{S1E1}).
\begin{proof}[Proof of Theorem~\ref{T2}]
First, the existence and uniqueness of $\tau_n$ is proved in Lemma~\ref{L2}~(i).
The exact expression (\ref{T1E-1}) is also obtained in Lemma~\ref{L2}~(i).\\
(i) Let $0<\tau<\tau_n$.
It follows from Lemma~\ref{L3}~(ii) that
\[
\sigma_{\calL}\cap(i\R\cup \calH_+)\subset 
\left\{\mu^{(n)}_{n-1},\ldots,\mu^{(n)}_1\right\}.
\]
By Lemma~\ref{L1+}~(iv), there exists an eigenvalue $\mu^{(n)}_1\in\sigma_{\calL}$.
Since $\mu^{(n)}_1>0$, $(u^{\pm}_n,0)$ are unstable.
However, all the positive eigenvalues of $\sigma_{\calL}$ are of order $O(\e^{-C/\e})$, because $\mu_1^{(n)},\ldots,\mu_{n-1}^{(n)}$ are of order $O(\e^{-C/\e})$ (Proposition~\ref{S2P1}~(iii)).
Thus, $(u^{\pm}_n,0)$ are metastable.
The proof of (i) is complete.\\
(ii) We study $\sigma_g$.
The limit equation of $\gamma g(\lambda,\tau)/\tau=0$ as $\tau\to\infty$ becomes
\[
\lambda\{\lambda^2+2\lambda-3(1-\rho^2)^2\}=0,
\]
and hence $\lambda=0,\mu^{(n)}_0,\mu^{(n)}_n$.
Here, $\mu^{(n)}_0$ and $\mu^{(n)}_n$ are defined in Proposition~\ref{S2P1}~(i).
Hence, if $\tau>0$ is large, then the equation $g(\lambda,\tau)=0$ has three real roots which are close to $0,\mu^{(n)}_0,\mu^{(n)}_n$.
Thus, it follows from Lemma~\ref{L1} that all eigenvalues are real.
Suppose that $\lambda>\mu_0^{(n)}(>0)$.
By \eqref{L1E3++} we have
\[
0<\tau\lambda+\gamma=\alpha\beta\frac{-3(1-\left<(u^{\pm}_n)^2,1\right>)-\lambda}{(\lambda-\mu^{(n)}_0)(\lambda-\mu^{(n)}_n)}<0,
\]
which is a contradiction.
Therefore, every root is less than or equal to $\mu^{(n)}_0$.
By Lemma~\ref{L1} we have
\[
\sigma_{\calL}\subset \sigma_L\cup \sigma_g\cup\sigma_0 \subset(-\infty,\mu^{(n)}_0].
\]
(iii) When $|\tau-\tau_n|$ is small, the existence of one pair of complex conjugate eigenvalues $\lR(\tau)\pm i\lIn(\tau)$ follows from Lemma~\ref{L2}~(ii).
Other eigenvalues are real, because of Lemma~\ref{L1+}~(ii).
By (\ref{L1E0}), they are on $(-\infty,\mu^{(n)}_1]$, since $\mu^{(n)}_0\not\in\sigma_{\calL}$ (Lemma~\ref{L1+}~(v)).
Other assertions on $\lR(\tau)\pm i\lIn(\tau)$ follow from Lemma~\ref{L2}~(ii) and (iii), and the exact expression (\ref{T1E1}) is obtained in Lemma~\ref{L2}~(i).
A pair of complex conjugate eigenvalues is simple (Lemma~\ref{L2+}).
They cross the imaginary axis at $\tau=\tau_n$ (Lemma~\ref{L2}~(ii)) with non-zero speed (Lemma~\ref{L2}~(iii)).
There is no other complex eigenvalue (Lemma~\ref{L1+}~(ii)).
Therefore, (\ref{S5T1E6}) and (\ref{S5T1E7}) are satisfied.
It follows from Theorem~\ref{S5T1} that the assertions (iii) hold.\\
(iv) The formulas (\ref{T2F1}), (\ref{T2F2}) and (\ref{T2F3}) follow from (\ref{L4E1}) (\ref{L4E2}) and (\ref{L4E3}), respectively.\\
The proof is complete.
\end{proof}

\begin{proof}[Proof of Theorem~\ref{T1} ]
The proof of Theorem~\ref{T1} is almost the same as that of Theorem~\ref{T2}.
The different parts are about $\sigma_{\calL}\cap (i\R\cup\calH_+)$.
When $n=1$, $\sigma_{\calL}\cap (i\R\cup\calH_+)=\emptyset$ for $0<\tau<\tau_1$ (Lemma~\ref{L3}~(ii)), and hence the assertion (i) holds.
Because of Lemma~\ref{L3}~(iii), the assertion (ii) holds.
All the eigenvalues except a pair of complex conjugate eigenvalues are real and negative.
Hence, all the assertions in (iii) hold.
The proof is complete.
\end{proof}

\section{Antiphase horizontal oscillation}
We show that an anti-phase horizontal oscillation of the layers occurs near the Hopf bifurcation point, using a formal argument.

Let $(u(x,t;r),\xi(t;r),\tau(r))$, $|r|<r_0$, be a continuously differential curve of $2\pi/\kappa(r)$-periodic solutions of (\ref{S1E1}), which is obtained by Theorems~\ref{T1} and \ref{T2}, such that
\[
(u(x,t;0),\xi(t;0),\tau(0))=(u^{\pm}_n(x),0,\tau_n)\ \ \textrm{and}\ \ \frac{2\pi}{\kappa(0)}=\frac{2\pi}{\lIn}.
\]
By an abstract Hopf bifurcation theorem~\cite[Theorem I.9.1]{K04} we see that
\[
\left.\frac{d}{dr}u(x,t;r)\right|_{r=0}=2{\rm Re}\left(\phi e^{i\kappa(0)t}\right),
\]
where $(\phi,\eta)$ is an eigenvector of (\ref{EVP}) with respect to the eigenvalue $i\lIn$ and $\kappa(0)=\lIn$.
By (\ref{T2F1}) we see that
\begin{equation}\label{S6E1}
\lIn=O(\sqrt{\e}e^{-\frac{1}{\sqrt{2}n\e}}).
\end{equation}
By Proposition~\ref{S2P4} we see that
\begin{equation}\label{S6E2}
1-\rho^2=\frac{1-k^2}{1+k^2}=O(e^{-\frac{1}{\sqrt{2}n\e}}).
\end{equation}
By (\ref{S6E1}), (\ref{S6E2}) and Proposition~\ref{S2P2} we have
\begin{align}
\phi &=\frac{\lIn^3\left\{1-3(u^{\pm}_n)^2\right\}+i\lIn^2\left\{\lIn^2+3(1-\rho^2)^2+6-6(u^{\pm}_n)^2\right\}}{\left\{\lIn^2+3(1-\rho^2)^2\right\}^2+4\lIn^2}
\nonumber\\
&=i\frac{3}{2}\left\{1-(u^{\pm}_n)^2\right\}+O(\sqrt{\e}e^{-\frac{1}{\sqrt{2}n\e}}).\label{S6E3}
\end{align}
Since
$\left\{1-(u^{\pm}_n)^2\right\}^2=2\e^2\left\{(u^{\pm}_n)'\right\}^2+(1-\rho^2)^2$, 
we have
\begin{equation}\label{S6E4}
1-(u^{\pm}_n)^2=\sqrt{2}\e|(u^{\pm}_n)'|+O(e^{-\frac{1}{\sqrt{2}n\e}}).
\end{equation}
By (\ref{S6E3}) and (\ref{S6E4}) we have
\begin{equation}\label{S6E5}
2{\rm Re}(\phi e^{i\kappa(0)t})=-3\sqrt{2}\e|(u^{\pm}_n)'(x)|\sin(\lIn t)+O(e^{-\frac{1}{\sqrt{2}n\e}}).
\end{equation}
On the other hand, if $|r|$ is small, then the Taylor expansion in $r$ yields
\begin{align}
&u^{\pm}_n\left(x-3\sqrt{2}\e\,\sgn\!\left((u^{\pm}_n)'(x)\right)\sin(\lIn t)r\right)\nonumber\\
&\qquad =u^{\pm}_n(x)-3\sqrt{2}\e(u^{\pm}_n)'(x)\,\sgn\!\left((u^{\pm}_n)'(x)\right)\sin(\lIn t)r+o(r)\nonumber\\
&\qquad =u^{\pm}_n(x)-3\sqrt{2}\e |(u^{\pm}_n)'(x)|\sin(\lIn t)r+o(r).\label{S6E6}
\end{align}
By (\ref{S6E5}) and (\ref{S6E6}) we formally obtain
\[
u(x,t;r)\simeq u^{\pm}_n\left(x-3\sqrt{2}\e\,\sgn\!\left((u^{\pm}_n)'(x)\right)\sin(\lIn t)r\right).
\]
Since each layer of $u^{\pm}_n(x)$ is located at
\[
x=\frac{1+2\ell}{2n}\ \ \textrm{for}\ \ \ell=0,1,\ldots,n-1,
\]
we have
\[
x\simeq \frac{1+2\ell}{2n}+3\sqrt{2}\e\,{\rm sgn}\!\left((u^{\pm}_n)'(x)\right)\sin(\lIn t)r.
\]
We formally obtain a layer oscillation.
Specifically, a position of each layer of $u(x,t;r)$ is
\[
x\simeq \frac{1+4m}{2n}\mp 3\sqrt{2}\e\sin(\lIn t)r\ \ \textrm{for $m=0,1,\ldots,\left[\frac{n}{2}-\frac{1}{4}\right]$}
\]
and
\[
x\simeq \frac{3+4m}{2n}\pm 3\sqrt{2}\e\sin(\lIn t)r\ \ \textrm{for $m=0,1,\ldots,\left[\frac{n}{2}-\frac{3}{4}\right]$.}
\]
Here, $[\xi]$ indicates the largest integer that does not exceed $\xi$.
Therefore, an anti-phase horizontal oscillation of layers occurs near the Hopf bifurcation point.

\section{Numerical simulations}
In this section we show numerical simulations of solutions of the time evolution system (\ref{S1E1}) with initial data
\[
(u(x,0),\xi(0))=(u_0(x),\xi_0),
\]
where $(u_0(x),\xi_0)$ is near the stationary solution $(u_1^-,0)$ or $(u_2^-,0)$.
We set the parameters in (\ref{S1E1}) as follows:
\[
\alpha=0.5,\quad\beta=0.5,\quad\gamma=0.5
\]
and $\e$ and $\tau$ are given later.
The interval domain $(0,1)$ is discretized into $200$ equidistant points, {\it i.e.,} $\Delta x=0.005$.
The time is discretized as $t=n\Delta t$, $n=0,1,2,\cdots$, where $\Delta t=0.01$.
Figure~\ref{fig1} shows two initial data $(u_0(x),\xi_0)$.
\begin{figure}[ht]
 \includegraphics[keepaspectratio,width=16cm]{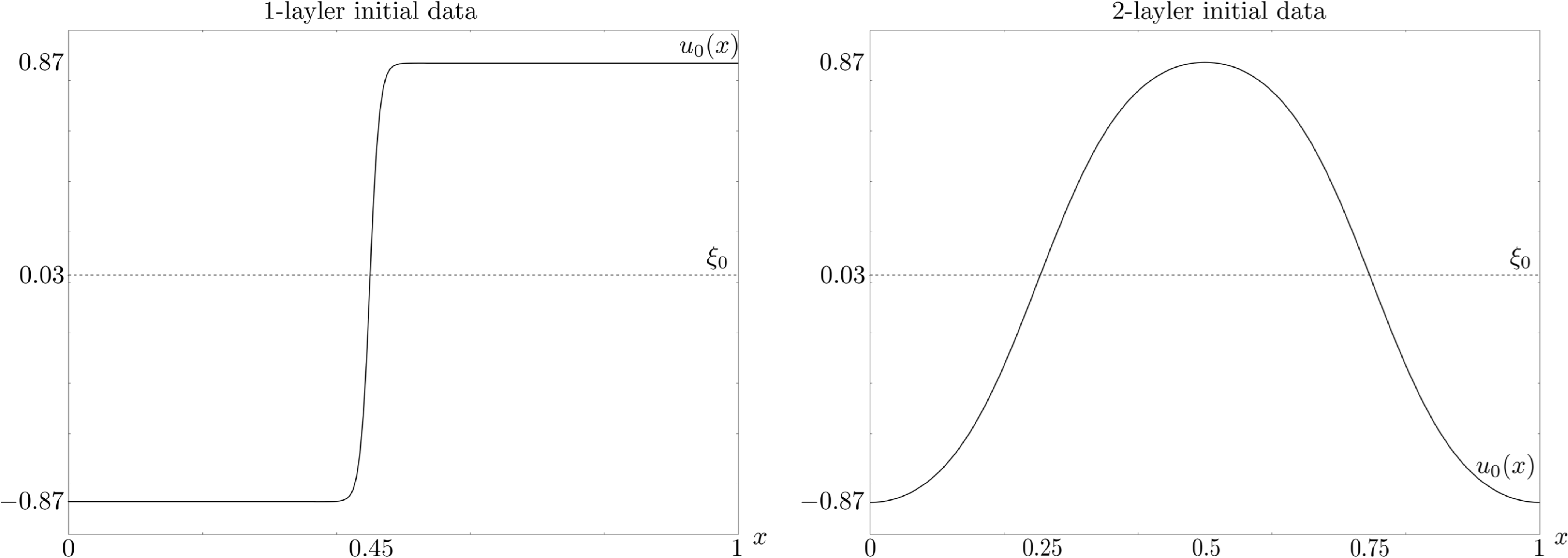}
 \caption{Two initial data.
The left function is used as an initial function $u_0$ in (A1) and (B1) in Figure~\ref{fig2}.
The right function is used as an initial function $u_0$ in (C1) and (D1) in Figure~\ref{fig2}.
In both initial data, $\xi_0=0.03$.
}\label{fig1}
\end{figure}
One is a function that is the $1$-layer stationary solution $(u_1^-,0)$ plus a positive perturbation (left),
and the other is close to the $2$-layer stationary solution $(u_2^-,0)$ (right).
In both initial data $\xi_0=0.03$.
Figure~\ref{fig2} shows bird views of four numerical solutions of (\ref{S1E1}).
Specifically, the graph of $u(x,t)$ is shown, but that of $\xi(t)$ is omitted.
The horizontal axis stands for the $x$-axis and the depth axis stands for the $t$-axis.
The parameters $\e$ and $\tau$ are chosen as Table~\ref{tab1} shows.
\begin{table}[ht]
\caption{$(\e,\tau)$ is chosen as follows:}\label{tab1}
\begin{tabular}{|c|cccc|}
\hline
 & (A1) & (B1) & (C1) & (D1)\\
\hline
$\e$ & $0.2$ & $0.2$ & $0.1$ & $0.1$\\
$\tau$ & $6.0$ & $6.7$ & $6.0$ & $6.7$\\
\hline
\end{tabular}
\end{table}
\begin{figure}[ht]
 \centering
 \includegraphics[keepaspectratio,width=16cm]{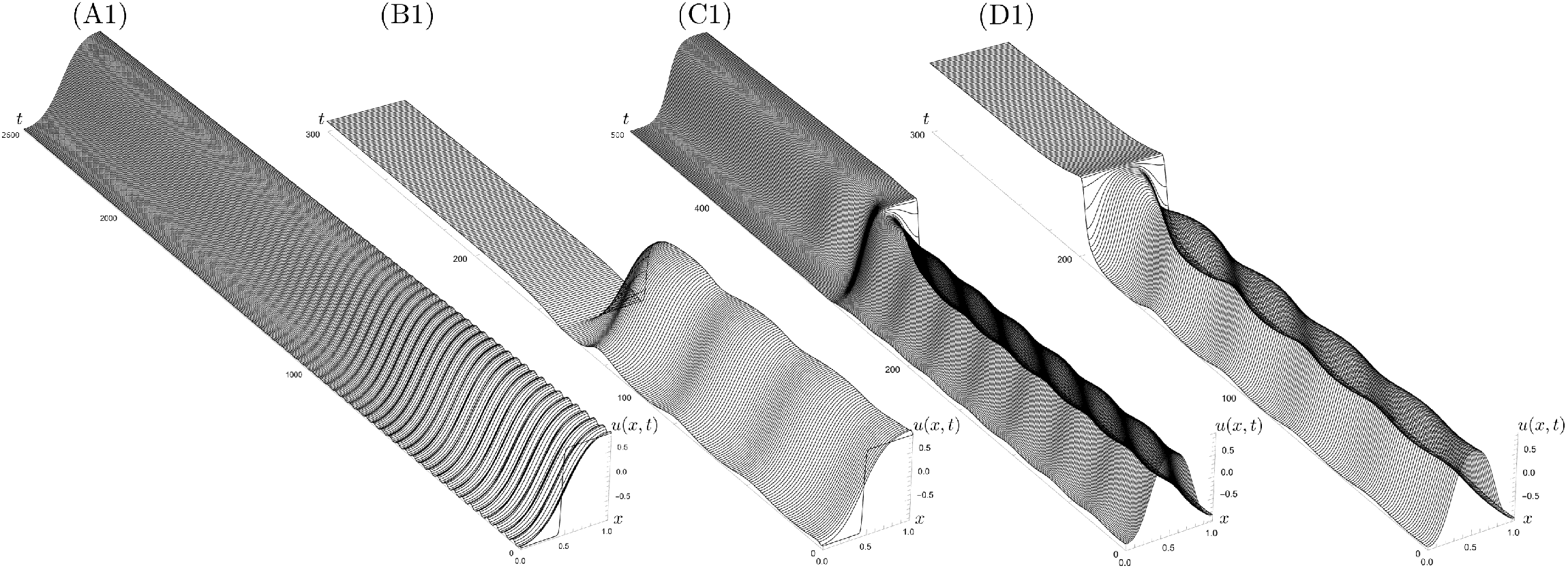}
 \caption{Bird views of $u(x,t)$. $(u_1^-,0)$ is stable in (A1), but it is unstable in (B1).
In both (C1) and (D1) $(u_2^-,0)$ is unstable. Anti-phase horizontal oscillations of two layers are observed in (C1) and (D1).
}\label{fig2}
\end{figure}

In (A1) and (B1) of Figure~\ref{fig2} we choose the left function of Figure~\ref{fig1} as initial data.
In these two cases the exact critical value for the Hopf bifurcation is
\[
\tau_1=6.3828409010676614...\ \ \textrm{for}\ \ \e=0.2,
\]
which can be calculated by \eqref{e} and \eqref{T1E-1}.
Since
\[
6.0<\tau_1<6.7,
\]
it follows from Theorem~\ref{T1} that $(u_1^-,0)$ is stable in (A1) and is unstable in (B1).
Indeed, in (A1) the solution converges to $(u_1^-,0)$ for large $t>0$, while the solution in (B1) does not converge to $(u_1^-,0)$.
The single-layer oscillates in (A1), because of two stable complex eigenvalues which are eigenvalues with negative real parts.
Since the solution in (B1) does not approach to a certain periodic solution, the Hopf bifurcation from $(u_1^-,0)$ at $\tau=\tau_1$ is possibly subcritical.
In (C1) and (D1) of Figure~\ref{fig2} we choose the right function of Figure~\ref{fig1} as initial data.
Since the value of $\e$ in (C1) and (D1) is the half of that in (A1) and (B1), the exact critical value for the Hopf bifurcation in (C1) and (D1) is the same as that in the previous case, {\it i.e.,}
\[
\tau_2=6.3828409010676614...\ \ \textrm{for}\ \ \e=0.1.
\]
For this reason, see \eqref{e} and \eqref{T1E-1}.
In (C1) and (D1) Theorem~\ref{T2} says that $(u_2^-,0)$ is unstable for all $\tau>0$.
The solution of (C1) or (D1) does not converge to $(u_2^-,0)$.
In (C1) $(u_1^-,0)$ is stable and the solution in (C1) converges to $(u_1^-,0)$ for large $t>0$.
Note that $(u_1^-,0)$ is unstable in (D1) and that the solution converges to $(1/\sqrt{2},0)$ which is another stable stationary constant solution.
Anti-phase horizontal oscillations of two layers, which is explained by formal calculation in Section~6, are observed for small $t>0$ in (C1) and (D1).

The graphs (A2), (B2), (C2) and (D2) in Figure~\ref{fig3} show the graphs of $\int_0^1u(x,\,\cdot\,)dx$ in the cases (A1), (B1), (C1) and (D1), respectively.
\begin{figure}[ht]
 \centering
 \includegraphics[keepaspectratio,width=16cm]{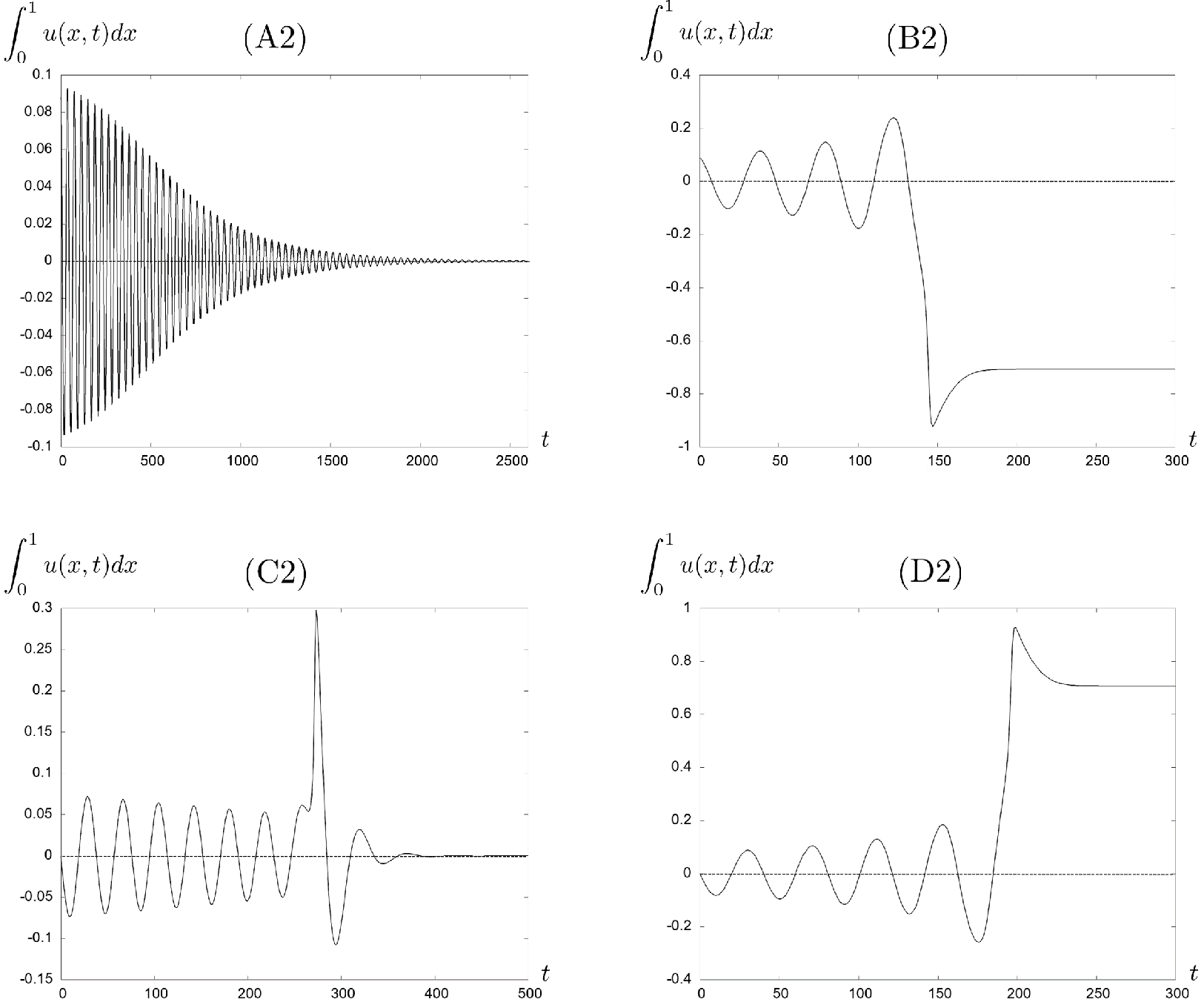}
 \caption{A graph of $\int_0^1u(x,\,\cdot\,)dx$ in four cases.
}\label{fig3}
\end{figure}
In (A2) and the first stage of (C2) amplitudes are decreasing, since $(u_1^-,0)$ is stable in (A1) and $(u_2^-,0)$ is metastable in (C1).
On the other hand, in (B2) and (D2) amplitudes are increasing, since $(u_1^-,0)$ in (B1) is unstable and $(u_2^-,0)$ in (D1) is unstable.
Figure~\ref{fig4} shows an enlarged view of (A2) for $0<t<100$.
\begin{figure}[ht]
 \centering
 \includegraphics[keepaspectratio,width=10cm]{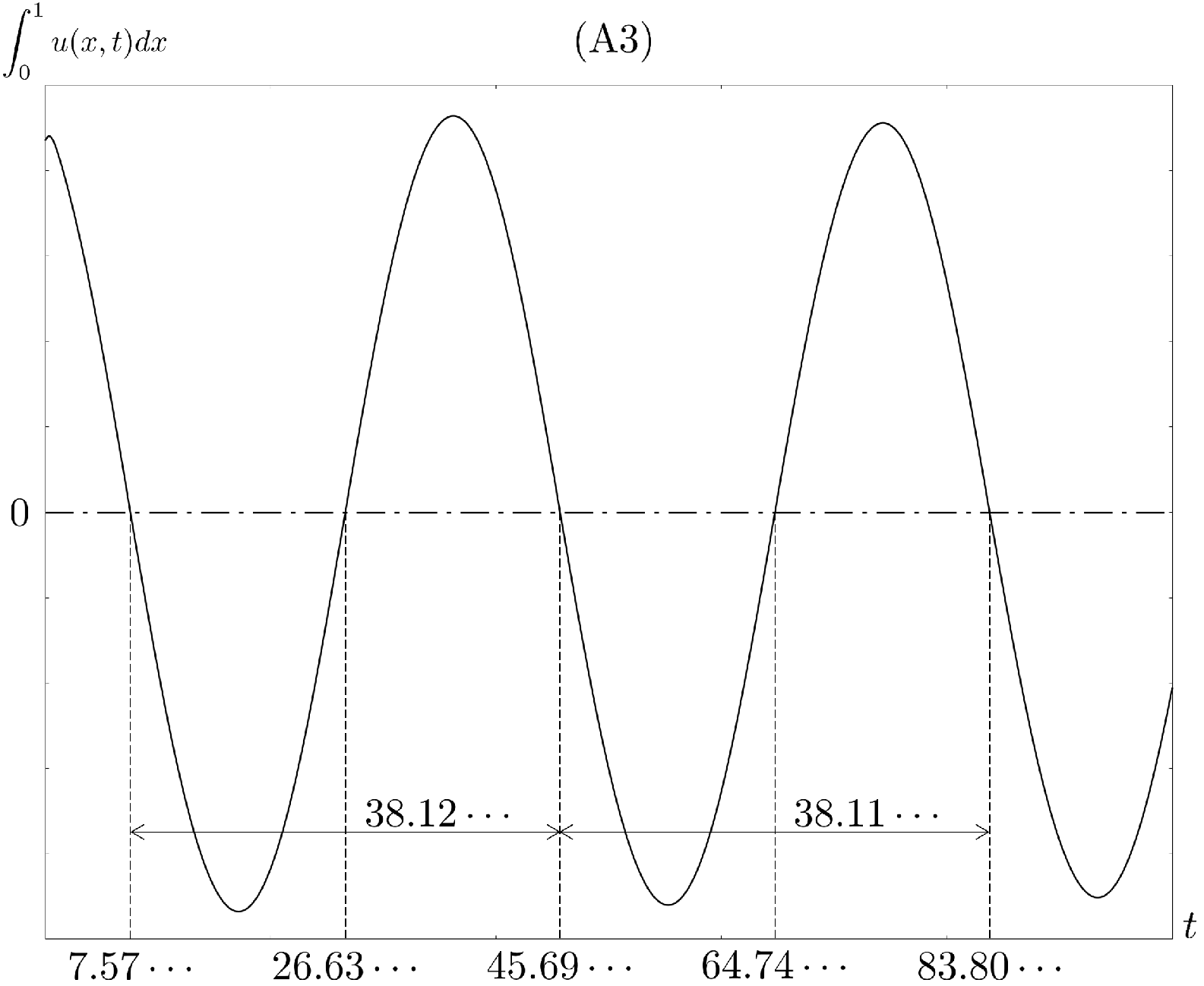}
 \caption{Enlarged view of (A2) for $0<t<100$. The solution is nearly periodic and the first period is approximately $38.12$. 
}\label{fig4}
\end{figure}
The first period of the nearly periodic solution is approximately $38.12$.
The exact period given by \eqref{T1E1} is
\[
\frac{2\pi}{\lambda_{{\rm I},1}}=38.929920651496360...
\]
Therefore, the relative error is
\begin{equation}\label{S7E2}
1-\frac{\textrm{Numerical period}}{\textrm{Exact period}}
\simeq 1-\frac{38.12}{38.92992}\simeq 2.1\%.
\end{equation}
It is difficult to observe oscillations numerically when $\tau$ is quite close to $\tau_1\simeq 6.38284$.
We have chosen $\tau=6.0$ in (A3) (and hence (A1) and (A2)).
Therefore, the relative error \eqref{S7E2} includes not only rounding error but also the error coming from $\tau\neq \tau_1$.

Figure~\ref{fig5} shows a graph of the relative error percentage, which is the LHS of \eqref{S7E2}, of the solution in (A2).
\begin{figure}[ht]
 \centering
 \includegraphics[keepaspectratio,width=10cm]{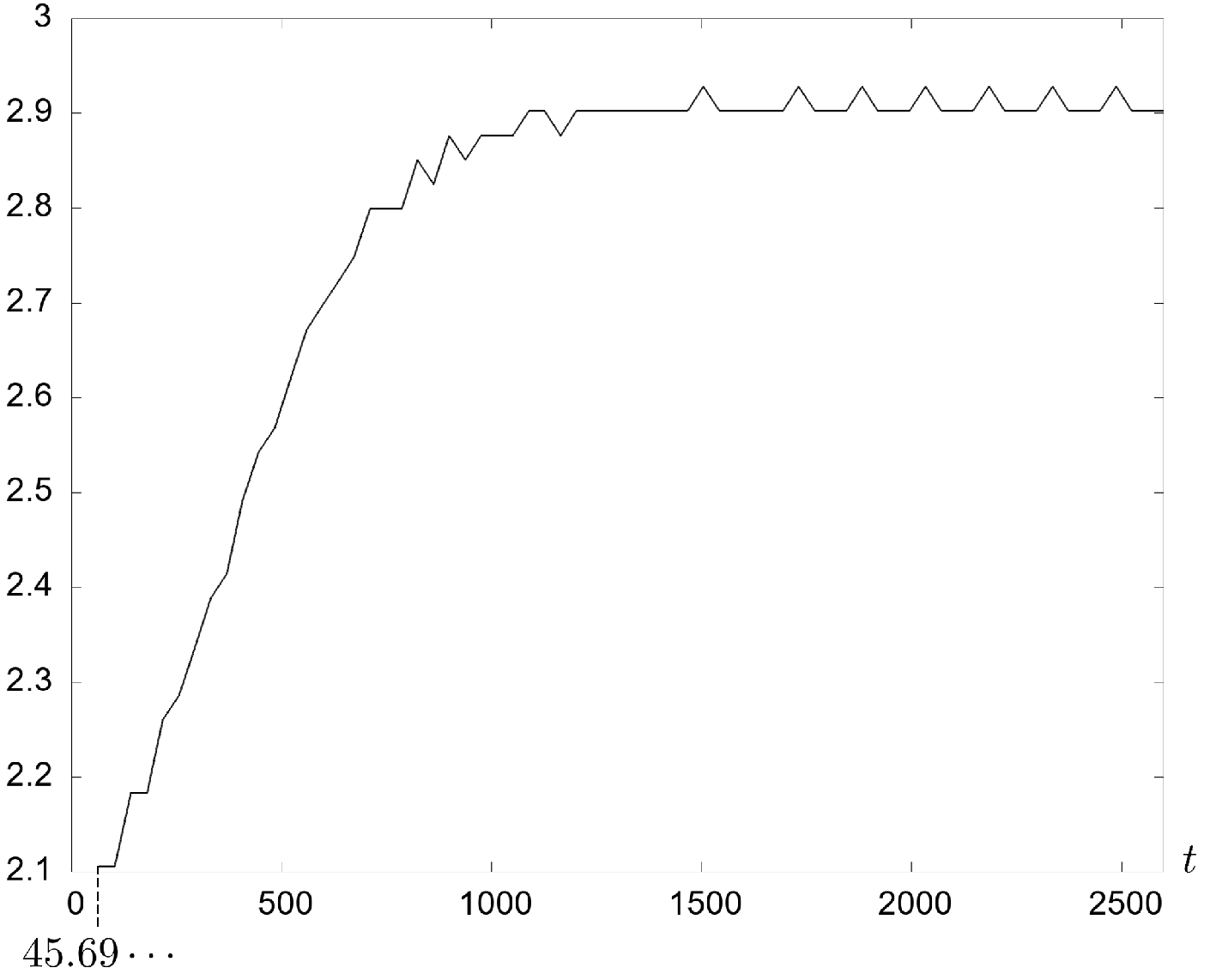}
 \caption{The relative error percentage of the numerical period of (A2) {\it i.e.,} $1-\textrm{Numerical period}/\textrm{Exact period}$.
The relative error starts from $2.1\%$ near $t=0$ and it is increasing in $t$ up to $2.9\%$.
}\label{fig5}
\end{figure}
The horizontal axis stands for the $t$-axis and the vertical axis stands for the relative error percentage.
The graph starts from $2.1\%$ near $t=0$, it increases in $t$ and it is around $2.9\%$ for $1000<t<2600$.
As mentioned above, the Hopf bifurcation is possibly subcritical.
In that case an unstable periodic solution exists in (A1), since $\tau=6.0<\tau_1\simeq 6.38284$.
It seems that the solution in (A1) gradually moves away from the unstable periodic solution, since the solution is attracted by the stable stationary solution $(u_1^-,0)$.
The fact that the numerical period in (A1) is less than the exact period agrees with \eqref{L2E7}, since $\tau=6.0<\tau_1\simeq 6.38284$.

\section{Conluding remarks}
In this paper, we constructed a single and multi-layer stationary solutions and
showed the occurrence of the Hopf bifurcation in exact ways.
Numerical simulations suggested the subcritical bifurcation structure
of  periodic solutions
and also the  anti-phase horizontal oscillations of two layers was explained
in a formal way.
They are problems to be solved as future works. But the
direct calculations by using exact solutions seem so hard.
One way to overcome the difficulty is a contraction technique of the
dynamics such as front interaction. Then the motions can be reduced
to a certain finite dimensional ordinary differential systems,
which are expected to be treated with less difficulty.
In practice, the motion of a single layer solution for a shadow system \rf{S1E1}
is reduced to a two dimensional ODE and the subcritical bifurcation structure
in the neighborhood of the Hopf bifurcation point should be analyzed 
by the reduced ODE.
We will give results on it in a forthcoming paper \cite{EMM}.

\newpage
\renewcommand{\thesection}{\Alph{section}}
\def\theequation{\Alph{section}.\arabic{equation}}
\setcounter{section}{0}
\section{Appendix: Elliptic integrals and functions}
\subsection{Complete elliptic integrals}
Let $k\in[0,1)$.
The complete elliptic integral of the first kind $K(k)$ is defined by
\begin{equation}\label{K}
K(k):=\int_0^1\frac{ds}{\sqrt{(1-s^2)(1-k^2s^2)}},
\end{equation}
and it satisfies $K(k)=\sn^{-1}(1,k)$.
The complete elliptic integral of the second kind is defined by
\begin{equation}\label{E}
E(k):=\int_0^1\sqrt{\frac{1-k^2s^2}{1-s^2}}ds.
\end{equation}
The function $K(k)$ is monotonically increasing in $k$,
\[
K(0)=\frac{\pi}{2},\ \ \lim_{k\to 1^-}K(k)=\infty
\]
and $E(k)$ is monotonically decreasing in $k$,
\[
E(0)=\frac{\pi}{2},\ \ \lim_{k\to 1^-}E(k)=1.
\]

\subsection{Elliptic functions}
Let $k\in (0,1)$.
The Jacobi elliptic function $\sn(x,k)$ is an odd, $4K(k)$-periodic, $2K(k)$-antiperiodic and analytic function for $x\in\R$ and the inverse function $\sn^{-1}(y,k)$ is defined locally by
\[
\sn^{-1}(y,k)=\int_0^{y}\frac{ds}{\sqrt{(1-s^2)(1-k^2s^2)}}
\]
for $0\le y\le 1$.
The function $y=\sn(x,k)$ is a solution of
\[
\begin{cases}
y''+(1+k^2)y-2k^2y^3=0 & \textrm{for}\ x\in\R,\\
y(0)=0,\ y'(0)=1.
\end{cases}
\]
\smallskip

\subsection{Formulas}
We give standard formulas in Lemmas~\ref{AL1}--\ref{AL3} without proofs.
See \cite{BF71,YMM10} for details.
\begin{lemma}\label{AL1}
Let $k\in (0,1)$. Then,
\[
(1-k^2)K(k)<E(k)<\left(1-\frac{1}{2}k^2\right)K(k).
\]
\end{lemma}
See \cite[p.128]{YMM10} for Lemma~\ref{AL1}

\begin{lemma}\label{AL2}
Let $k\in(0,1)$. Then,
\[
\lim_{k\to 1^-}\left( K(k)-\log\frac{1}{\sqrt{1-k^2}}-2\log 2\right)=0.
\]
\end{lemma}
See \cite[Eq.112.01 in p.11]{BF71} for Lemma~\ref{AL2}.

\begin{lemma}\label{AL3}
Let $k\in (0,1)$. Then,
\[
\int_0^{K(k)}\sn^2(x,k)dx=\frac{K(k)-E(k)}{k^2}.
\]
\end{lemma}
See \cite[Eq.310.02 in p.191]{BF71} for Lemma~\ref{AL3}.


\end{document}